\documentclass[10pt]{article}
\usepackage{listings,amsmath,amssymb,xcolor,graphicx,url}
\usepackage{amsthm}
\usepackage[inner=3cm,outer=3cm,top=3cm,bottom=3cm]{geometry}
\usepackage{appendix}
\usepackage{tikz}
\usepackage{subcaption}
\definecolor{lgrey}{HTML}{808080}
\definecolor{dgrey}{HTML}{404040}
\usepackage{algorithm}
\usepackage[noend]{algpseudocode}
\def\input@path{{Numerical results from joint reconstruction/}}
\newtheorem{theorem}{Theorem}
\newtheorem{lemma}{Lemma}
\newtheorem{corollary}{Corollary}
\newtheorem{remark}{Remark}
\newtheorem{defn}{Definition}

\newcommand{\TV}{\mathrm{TV}}

\newcommand{\R}{\mathbb{R}}
\newcommand{\fL}{\mathrm{L}_{\beta}}

\DeclareMathOperator*{\argmin}{argmin}

\title{Emission tomography with a multi-bang assumption on attenuation} 

\author{Sean Holman and Philip Richardson} 

\begin{document}
\maketitle
\abstract{We consider the problem of joint reconstruction of both attenuation $a$ and source density $f$ in emission tomography in two dimensions. This is sometimes called the Single Photon Emission Computed Tomography (SPECT) identification problem, or referred to as attenuation correction in SPECT. Assuming that $a$ takes only finitely many values and $f \in C_c^1(\mathbb{R}^2)$ we are able to characterise singularities appearing in the Attenuated Radon Transform $R_a f$, which models emission tomography data. Using this characterisation we prove that both $a$ and $f$ can be determined in some circumstances. We also propose a numerical algorithm to jointly compute $a$ and $f$ from $R_af$ based on a weakly convex regularizer when $a$ only takes values from a known finite list, and show that this algorithm performs well on some synthetic examples.}
\section{Introduction}

In this paper we consider the problem of emission tomography in which we seek to recover both attenuation $a$ and radiation source density $f$ based on passive radiation measurements. We will only consider the problem in two dimensions, and use the photon transport equation for the forward model. Assume that $a$ and $f$ are compactly supported functions on $\mathbb{R}^2$ and let $\Omega\subset \R^{2}$ be a bounded set containing the supports. Then the photon transport equation is 
\begin{equation}\label{phtranseq}
\begin{split}
\theta\cdot\nabla{u(x,\theta)}+a(x)u(x,\theta) = f(x), \quad (x,\theta)\in \Omega\times\mathbb{S}^{1},\\
u|_{\Gamma^{-}}=0,
\end{split}
\end{equation}
where $u(x,\theta)$ is the photon flux through the point $x$ in a unit direction $\theta\in\mathbb{S}^{1}$ and 
$$ \Gamma^{-}=\{(x,\theta)\in\partial\Omega\times\mathbb{S}^{1}\ |\ \theta\cdot \underline{n}(x)\leq 0\},$$
where $\underline{n}(x)$ is the unit outward pointing normal to the boundary at $x$. Intuitively the differential equation in \eqref{phtranseq} states that photons are created by a source with density $f$ and then move along straight lines while being attenuated at a rate given by $a$. The boundary condition in \eqref{phtranseq} requires that we have no radiation entering the domain $\Omega$. The photon transport equation can be solved (for example by method of characteristics as in \cite{atrt}) to find the solution $u(x,\theta)$ which satisfies
\begin{equation}\label{eq,2} \lim_{\tau\rightarrow\infty}u(s \theta^\perp +\tau\theta,\theta)=\int_{-\infty}^{\infty}\!f(s \theta^\perp+t\theta)e^{-Da(s \theta^{\perp}+t\theta,\theta)}\,\mathrm{d}t,
\end{equation}
where $\theta^\perp\in \mathbb{S}^1$ denotes $\theta$ rotated in the anticlockwise direction by $\pi/2$ and $Da$ is the beam transform of $a$ defined as 
$$	Da(x,\theta)=\int_{0}^{\infty}\! a(x+t\theta)\,\mathrm{d}t.$$
We define the Attenuated Radon Transform(AtRT)\cite{atrt} of $f$ with attenuation $a$ to be the integral on the right hand side of \eqref{eq,2}, and denote this $R_af(s,\theta)$. The basic question we investigate in this paper is whether it is possible to determine both of $a$ and $f$ from $R_a f$.

When $a$ is fixed, the mapping $f \mapsto R_a f$ is known to be analytically invertible under certain mild conditions, and when these hold a closed form solution for the inverse is known, see \cite{inversion,Identification2,Identification3}. The problem of recovering both $a$ and $f$ from the AtRT is sometimes known as the Single Photon Emission Computed Tomography (SPECT) identification problem \cite{Ident,Identification1}, and one particular result, given in \cite{Quinto}, shows non-uniqueness for radial $a$ and $f$. That is, there are different pairs of $a$ and $f$ depending only on distance to the origin which give the same AtRT. The issue persists for maps which are also ``close" to being radial, as shown in \cite{GourionNoll}, and numerical investigations in \cite{Phil} also show evidence of non-uniqueness in other situations. One should also note that if $f = 0$, then it is trivially clear that $a$ can be anything and still $R_af = 0$.

Despite these negative results, under some additional hypotheses determination of $a$ and $f$ from $R_a f$ can still be possible. In medical imaging literature there has been quite a lot of work aimed towards numerical methods for what is more typically called ``attenuation correction" in that context (e.g. \cite{attcorr1,attcorr2,attcorr3} and many others), although for practical applications of SPECT the standard method is to obtain the attenuation through a separate transmission CT scan. Some studies have also attempted to make use of scattered photons for attenuation identification in SPECT (\cite{scatter1,scatter2}) although the forward model must be enriched to include the scattering, and so the mathematical problem is different. There are not many positive theoretical results concerning recovery of both $a$ and $f$ from emission data alone. In the case when $Da$ in \eqref{eq,2} is replaced by a constant $\mu$ times $t$ the transform is called the exponential Radon transform, and in \cite{Identification1} it is shown that $\mu$ can be determined from the exponential Radon transform when $f$ is unknown, but not radial. A linearisation of the problem is studied from a microlocal point of view in \cite{Ident}, and used to establish some results for the nonlinear problem as well. A range characterisation of $f \mapsto R_a f$ for given $a$ originally found in \cite{inversion} and further explored in \cite{Novikovrange} was used in \cite{BalJollivet} to analyse recovery of both $a$ and $f$. Perhaps closest to the results of the present paper is the work in \cite{Bukgeimstar} which shows that unique recovery of $a$ and $f$ is possible when $a$ is a multiple of the characteristic function of a star shaped region. In this work we assume that $a$ takes on only finitely many values, and refer to such $a$ as {\it multi-bang}. As detailed below, we are able to show unique recovery of such $a$ from $R_af$ in some cases.

Recently the authors of \cite{MB,MBorig} introduced a convex multi-bang regularization technique intended to allow reconstructions of images in which there are only certain known values, and our line of research leading to the present paper was originally inspired by this technique. There are many applications where the multi-bang regularization technique might be useful, particularly in many forms of medical imaging, e.g. SPECT imaging\cite{GourionNoll} and X-ray Imaging\cite{limrays}. The convex multi-bang technique of \cite{MB,MBorig} was applied numerically to the problem of recovering multi-bang $a$ and $f$ from $R_a f$ in \cite{Phil} with mixed results, and in the present paper we modified the method to use a weakly convex (rather than convex) multi-bang regularization combined with Total Variation(TV) to promote the joint recovery of multi-bang $a$ and $f$ from the AtRT $R_af$. We implement this by alternating updates between $a$ and $f$ using \cite{alternating} to prove convergence. The $a$ update is the most computationally intensive step due to the nonlinearity and using recent work by \cite{Guo,weakconvex} we apply a variant of the Alternating Direction Method of Multipliers(ADMM) with a non-convex multi-bang regularizer, which we show lends itself to promoting multi-bang solutions. 

The novel contributions of this paper are summarised as follows. The precise and rigorous versions of the theoretical results, which include a few other technical assumptions, are presented later in the paper.
\begin{enumerate}
	\item Theorem \ref{thm1}. Assuming that $a$ is multi-bang, $f \in C^1_c(\mathbb{R}^2)$ is non-negative, and with some additional assumptions about the regularity of the boundaries of the regions of constant $a$, we can characterise the singularities occurring in $R_af$ as a result of the jumps in $a$.
	
	\item Theorem \ref{thm2}. If $a$ and $f$ are as in Theorem \ref{thm1} with the regions of constant $a$ being constructed using a sequence of nested convex sets, then $a$ and $f$ can be uniquely determined from $R_af$.
			
	\item We propose a numerical algorithm for joint recovery of multi-bang $a$ and $f$ for limited projection data, and demonstrate its utility with some numerical examples.
\end{enumerate}
Our proofs for Theorems \ref{thm1} and \ref{thm2} are based on a careful analysis of the singularities which occur in $R_af$ arising from the jumps of $a$, as well as a result that if $a$ is known outside of a convex region, then $R_af$ determines $f$ uniquely also outside this convex region (see Lemma \ref{lem:f}). We prove this latter result by reducing it to the problem considered in \cite[Theorem 3.1]{Bukgeim}, although other proofs using for example analytic microlocal analysis as in \cite{FSU} should also be possible.

The rest of the paper is structured as follows. Section \ref{sec:theory} gives the theoretical results of the paper introducing some necessary definitions in subsection \ref{sec:def}, stating and proving Theorem \ref{thm1} as well as some related results in subsection \ref{sec:thm1}, and stating and proving Theorem \ref{thm2} in subsection \ref{sec:thm2}.  Section \ref{sec:nummethod} describes the numerical methods used and section \ref{sec:numex} gives a number of numerical examples. The final section concludes the work in the paper and suggests avenues for further research.

\section{Identification problem theoretical results} \label{sec:theory}

\subsection{Problem set-up and definitions}\label{sec:def}

We first recall from the introduction that the AtRT, $R_af$ is defined by the formula
\begin{equation}\label{start}
R_{a}f(s,\theta)=\int_{-\infty}^{\infty}\!f(s\theta^{\perp}+t\theta)e^{-Da(s\theta^{\perp}+t\theta,\theta)}\,\mathrm{d}t.
\end{equation}
Throughout this section we will assume that $f \in C_c^1(\mathbb{R}^2)$ and that $a$ is multi-bang. Note that $(s,\theta)$ corresponds to the line tangent to $\theta$ with distance $s$ from the origin (this is slightly different from the standard parametrisation of lines in the Radon transform in which $\theta$ is normal to the line), and $R_af$ is an integral along that line. We will always parametrise $\theta \in \mathbb{S}^1$ by the angle $\omega \in \mathbb{R}$ such that $\theta = (\cos(\omega),\sin(\omega))$ and $\theta^\perp = (-\sin(\omega),\cos(
\omega))$.

We now introduce the precise definition of multi-bang.
\begin{defn}{\bf (Multi-bang)} \label{def:multibang}
We say that $a \in L^\infty(\mathbb{R}^2)$ is {\it multi-bang} if there exists a finite set $\mathcal{A} = \{a_1,\ ... \ , \ a_n\} \subset \mathbb{R}$, called the {\it admissible set}, and a collection of disjoint bounded open sets $\{\Omega_{j}\}_{j = 1}^n$ with smooth boundaries possibly having corners such that
\begin{equation} \label{multibang}
a = \sum_{j=1}^n a_j \chi_{\Omega_j}.
\end{equation}
Here $\chi_{\Omega_j}$ is the characteristic function of the set $\Omega_j$, and we assume that for all $\Omega_j$ the interior of the closure of $\Omega_j$ is equal to $\Omega_j$. We also assume that any line only intersects the boundaries $\cup_{j=1}^n \partial \Omega_j$ finitely many times. 
\end{defn}
\noindent The final hypothesis about lines intersecting the boundaries only finitely many times is added for technical reasons, and can probably be removed although we haven't proven this. It is also worthwhile to point out that for the theoretical results in section \ref{sec:theory} we do not require knowledge of the set of admissible values $\mathcal{A}$, although we do assume knowledge of the admissible values for the numerical algorithm. 

As we will see below in section \ref{sec:thm2}, we will be able to determine certain points on the boundaries of the sets $\Omega_j$ in Definition \ref{def:multibang} from $R_af$. The set of points which we can determine will be denoted $\mathcal{P}_a$, which we now define.
\begin{defn}{\bf ($\mathcal{P}_a$)}\label{def:P}
Suppose that $a$ is multi-bang with sets $\{\Omega_{j}\}_{j = 1}^n$ as in Definition \ref{def:multibang}. Then $\mathcal{P}_a$ is the set of points $x \in \partial \Omega_j$ for some $j$ such that either
\begin{enumerate}
\item\label{0curve} $\partial \Omega_j$ has non-zero curvature at $x$, or
\item\label{corner} $\partial \Omega_j$ has a corner at $x$.
\end{enumerate}
We further write $\mathcal{P}_{a,1}$ for the subset of $\mathcal{P}_a$ where the boundary has non-zero curvature as in \ref{0curve} and $\mathcal{P}_{a,2}$ for the subset of $\mathcal{P}_a$ of corners as in \ref{corner}.
\end{defn}
\noindent As we might expect from microlocal analysis (e.g. see \cite{Ident}) the jumps in $a$ along the boundaries of $\Omega_j$ lead to singularities in $R_af$ at points $(s,\theta)$ corresponding to lines which are either tangent to the boundaries $\partial \Omega_j$, or pass through corners of $\partial \Omega_j$. We introduce the following definition for these lines. Note that in this definition we consider lines passing through a corner of $\partial \Omega_j$ that are also tangent to one of the branches of the corner to be tangent to $\partial \Omega_j$ (so there would always be at least two lines passing through a corner that are also tangent).
\begin{defn}{\bf($ \mathcal{K}_a$)} \label{def:K}
Suppose that $a$ is multi-bang with sets $\{\Omega_{j}\}_{j = 1}^n$ as in Definition \ref{def:multibang} and $\mathcal{P}_a$ is as in Definition \ref{def:P}. We define $\mathcal{K}^0_a$ to be the subset of $\{ (s,\theta) \ : \ s \in \mathbb{R}, \quad \theta \in \mathbb{S}^1\}$ such that the line corresponding of $(s,\theta)$ is either tangent to $\partial \Omega_j$ or passes through a corner of $\partial \Omega_j$ for some $j$. We further define $\mathcal{K}_a \subset \mathcal{K}^0_a$ to be the same set with the added requirement that if the line is tangent, then the tangency must be at a point where $\partial \Omega_j$ has nonzero curvature. Finally, we define subsets of $\mathcal{K}_a$ which are the set $\mathcal{K}_{a,1}$ of lines tangent at a point where the curvature is non-zero, and $\mathcal{K}_{a,2}$ the set of lines passing through corners. (Note that $\mathcal{K}_{a,1}$ and $\mathcal{K}_{a,2}$ may not be disjoint.)
\end{defn}


\noindent For Theorem \ref{thm2} we will also require an additional definition which describes precisely what we meant when we said the regions of constant $a$ are constructed using a sequence of nested convex sets.

\begin{defn}{\bf(Nicely multi-bang)} \label{def:nicemb}
We say that $a$ is nicely multi-bang if $a$ is multi-bang and can furthermore be written in the form
\begin{equation}\label{eq:nicelymb}
a = \sum_{i=1}^N c_j \chi_{C_j}
\end{equation}
where the sets $C_j$ are all convex, bounded, open with smooth boundary possibly having corners and nested in the sense that
\[
C_n \Subset C_{n-1} \Subset \ ... \ \Subset C_1.
\]
Here $C_{j} \Subset C_{j-1}$ means that $C_{j}$ is contained in a compact set that is contained in $C_{j-1}$.
\end{defn}

\noindent In the next section we will state and prove Theorem \ref{thm1} as well as some other related results.

\subsection{Theorem \ref{thm1} and related results} \label{sec:thm1} 

We start the section with the statement of Theorem \ref{thm1}.

\begin{theorem} \label{thm1}
Suppose that $f \in C^1_c(\mathbb{R}^2)$ is non-negative and $a$ is multi-bang with sets $\Omega_j$ as given in Definition \ref{def:multibang}. The theorem has two parts corresponding to $\mathcal{P}_{a,1}$ and $\mathcal{P}_{a,2}$.
\begin{enumerate}
\item Suppose $x \in \mathcal{P}_{a,1}$ and the line tangent to a boundary $\partial \Omega_j$ at $x$ is not tangent to a boundary anywhere else. If this tangent line is given by $(s^*,\theta^*)$ with $\theta^* = (\cos(\omega^*),\sin(\omega^*))$ and the ray $\{x + t \theta^* \ | \ t < 0\}$ intersects the set $\{f>0\}$, then $\partial_s R_a f(s,\theta^*)$ has a singularity of order $1/2$ at $s = s^*$, and $x$ is the unique point on the line such that
\[
\lim_{\omega\rightarrow \omega^*}|\sin(\omega-\omega^*)|^{1/2} \partial_{\omega}R(x \cdot \theta^\perp,\theta) = 0.
\]
(Recall that $\theta = (\cos(\omega),\sin(\omega))$.)

\item Suppose that $x \in \mathcal{P}_{a,2}$ lies on a corner of a boundary for a component of $\Omega_j$ and is also a corner for only one other component of some $\Omega_l$. If $(s^*,\theta^*) \in \mathcal{K}_{a,2}$ corresponds to any line passing through $x$ that passes through no other corners, is not tangent to any of the boundaries and the ray $\{x + t \theta \ | \ t < 0\}$ passes through the set $\{f>0\}$, then $\partial_s R_a f(s,\theta^*)$ has a jump across $s = s^*$.
\end{enumerate}
\end{theorem}

\begin{remark}
We comment that while we have chosen Theorem \ref{thm1} to summarise the results in this section in a concise manner, in fact the lemmas we prove taken together can provide stronger statements and more information than given in Theorem \ref{thm1} concerning the singularities of $R_a f$ when $a$ is multi-bang. In particular Lemmas \ref{lem:cont} and \ref{lem:flat}, and Corollary \ref{cor:0curv} are not reflected in the statement of Theorem \ref{thm1}.
\end{remark}

\noindent We now begin establishing lemmas and a corollary which will lead to the proof of Theorem \ref{thm1}. For the first lemma we look at the regularity of $R_a f$ in the complement of $\mathcal{K}_a^0$, which is actually not directly related to Theorem \ref{thm1}.
. 
\begin{lemma}\label{lem:cont} Suppose that $f \in C_c^1(\mathbb{R}^2)$, $a$ is multi-bang, and $(s^*,\theta^*) \in (\mathcal{K}_a^0)^c$. Then the mapping $s \mapsto \partial_s R_a f(s,\theta^*)$ is continuous at $s^*$.
\end{lemma}
\begin{proof}
	Since $(s^*,\theta^*) \in (\mathcal{K}_a^0)^c$ the line tangent to $\theta^*$ with distance $s^*$ from the origin is neither tangent to one of the boundaries where $a$ jumps, nor passing through a corner of one of these boundaries. W.L.O.G we can rotate the axis so that the line corresponding to $(s^*,\theta^*)$ lies on the $x$-axis and $\theta^*=(1,0)$. Then we have 
\begin{equation}\label{lem1eq1}
R_{a}f(s,\theta^*)=\int_{-\infty}^{\infty}f(x,s)e^{-Da((x,s),\theta^*)}\mathrm{d}x.
\end{equation}
For $x\in\R$, recall that 
	$$Da((x,s),\theta^*)=\int_{x}^{\infty}a(t,s)\mathrm{d}t.$$
By assumption, the line given by $(s,\theta^*)$ only crosses the jumps of $a$ finitely many times and let us label the ordered values of $t$ for which these crossings occur (when the line is parametrised as $t \mapsto (t,s)$) as $\{t_i(s)\}_{i = 1}^N$. Since $(s^*,\theta^*) \in (\mathcal{K}_a^0)^c$ these functions $t_i$ are differentiable in a neighbourhood of $s = 0$. Next we introduce the functions
\begin{equation}\label{lem1phi}
\phi_i(x,s) = \left \{
\begin{array}{ll}
t_i(s) & x < t_i(s)\\
x & x \geq t_i(s).
\end{array}
\right .
\end{equation}
Note that for all $i$, $\phi_i $ is continuous with bounded first derivative in a neighbourhood of $s = 0$ which is also continuous when $x \neq t_i(s)$. Using these functions we have
\begin{equation} \label{lem1eq2}
Da((x,s),\theta^*)=\sum_{i=1}^{N-1}c_{i}(\phi_{i+1}(x,s)-\phi_{i}(x,s)) = \sum_{i=1}^N (c_{i-1}-c_i) \phi_i(x,s)
\end{equation}
where for each $i$, $c_i$ is one of the admissible values or possibly zero (in particular $c_0$ and $c_N$ are always zero). We thus see that $D a ((x,s),\theta^*)$ also has bounded derivatives in a neighbourhood of $s = 0$ that are continuous except when $x = t_i(s)$ for some $i$, and differentiating this with respect to $s$ gives
	$$\partial_{s}Da((x,s),\theta^*)=\sum_{i=1}^{N}(c_{i-1} -c_i)\partial_s \phi_i(x,s).$$ 
Since $f \in C^1(\mathbb{R}^2)$ as well, we can therefore differentiate under the integral sign in \eqref{lem1eq1} to get
\begin{equation}\label{lem1eq3}
\begin{split}
\partial_{s}R_{a}f(s,\theta^*) & =\int_{-\infty}^{\infty}\left(\partial_{s}f(x,s)-\partial_{s}Da((x,s),\theta^*)f(x,s)\right)e^{-Da((x,s),\theta^*)}\mathrm{d}x\\
& = \int_{-\infty}^{\infty}\partial_{s}f(x,s)e^{-Da((x,s),\theta^*)}\mathrm{d}x + \sum_{i=1}^{N} (c_i - c_{i-1}) t_i'(s) \int_{-\infty}^{t_i(s)} f(x,s) e^{-Da((x,s),\theta^*)}\mathrm{d}x.
\end{split}
\end{equation}
Since $f \in C^1_c(\mathbb{R}^2)$ and the derivatives $t_i'$ are all continuous, we see that in fact $\partial_s R_af(s,\theta^*)$ is also continuous with respect to $s$ at $s = 0$ which completes the proof.

\end{proof}

\noindent We now begin to consider what can happen to $R_a f$ at lines in $\mathcal{K}_a$. First we consider $\mathcal{K}_{a,2}$.

\begin{lemma}\label{lem:corner}
Suppose that $f \in C_c^1(\mathbb{R}^2)$, $a$ is multi-bang, and $(s^*,\theta^*) \in \mathcal{K}_{a,2}$ passes though exactly one corner and is not tangent to any boundary $\partial \Omega_j$. Then $s \mapsto \partial_s R_a f(s,\theta^*)$ is bounded near $s^*$. 

Suppose additionally that the corner point occurs at $s^* (\theta^*)^\perp + t^* \theta^*$, is a corner for $N$ different components of the regions $\Omega_j$, and the boundaries of these regions make angles $\{\alpha_k\}_{k=1}^N$ with $(\theta^*)^\perp$ where the orientation is chosen so that $\theta^*$ is at a positive angle. Also suppose that the jump in $a$ across the boundary with angle $\alpha_k$ in the direction of increasing angle is $b_k$ (see figure \ref{cornerfig} and caption). Then there is a jump in $\partial_s R_a f(s,\theta^*)$ across $s = s^*$ given by
\begin{equation} \label{eq:corner}
\Big [ \partial_s R_a f(s,\theta^*)\Big ]_{s^*_-}^{s^*_+} = \left ( \sum_{k=1}^N b_k \tan(\alpha_k) \right ) \int_{-\infty}^{t^*} f(s^* (\theta^*)^\perp + t \theta^*) e^{-Da(s^* (\theta^*)^\perp + t \theta^*, \theta^*)} \ \mathrm{d} t.
\end{equation}
\end{lemma}
\begin{figure}
\begin{center}
\includegraphics[scale=.7]{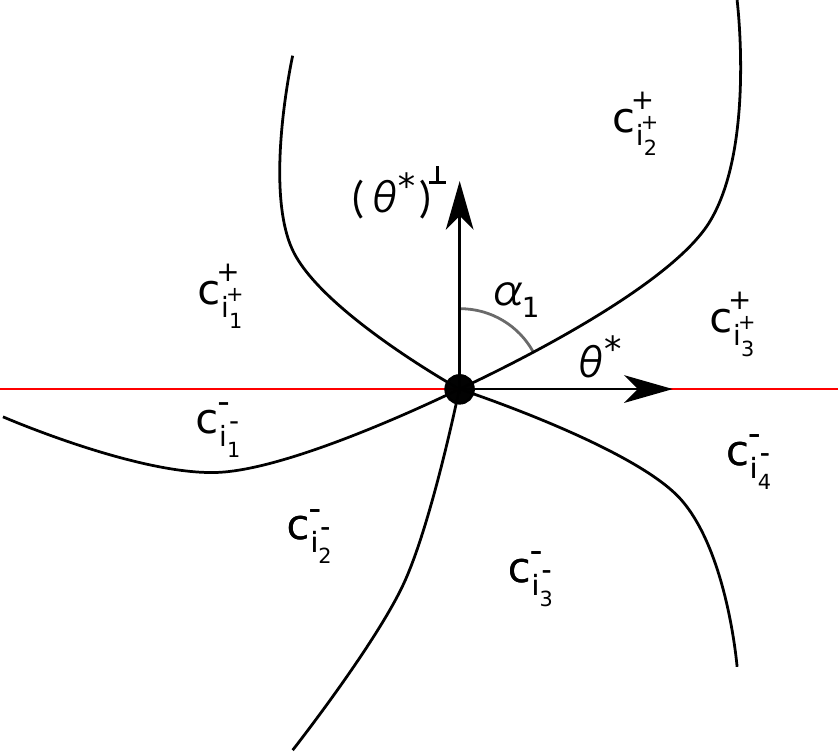}
\end{center}
\caption{This figure illustrates some of the notation used in the statement and proof of Lemma \ref{lem:corner}. The line corresponding to $(s^*,\theta^*)$ is shown in red. Note that the jump $b_1$ across the boundary corresponding to the angle $\alpha_1$ in the example shown in the figure would be $b_1 = c_{i_3^+}^+ - c_{i_2^+}^+$.}
\label{cornerfig}
\end{figure}
\begin{proof}
We use the same set-up and notation as in the proof of Lemma \ref{lem:cont}, although by translating if necessary we also assume W.L.O.G. that the corner occurs at the origin. Unlike in Lemma \ref{lem:cont} there may be a different number of boundary crossings $N$ when $s>0$ and $s<0$, and so we introduce corresponding functions $\{t^\pm_i\}_{i=1}^{N^\pm}$ giving the crossing points where the $t_i^+$ are defined for $s>0$ and the $t_i^-$ are defined for $s<0$. As in Lemma \ref{lem:cont} these $t_i^\pm$ will all have bounded derivatives up to $s = 0$ (this is because the line given by $(s^*,\theta^*)$ is not tangent to any boundary). We also introduce the corresponding $\phi^\pm_i$ defined as in \eqref{lem1phi} but only for $s>0$ and $s<0$ respectively. The formula \eqref{lem1eq2} still holds with $\pm$ added in appropriate places, and we can still see that $Da((x,s),\theta^*)$ is continuous for $s$ close to zero. For the derivative $\partial_s Da((x,s),\theta^*)$ we have for $s \neq 0$, where $\pm$ is the sign of $s$,
\begin{equation}\label{lem3eq1}
\partial_s Da((x,s),\theta^*)=\sum_{i=1}^{N^\pm}(c^\pm_{i-1}- c^\pm_i) \partial_s \phi^\pm_{i}(x,s).
\end{equation}
Since these derivatives are all bounded (but not necessarily continuous at $s = 0$) we still have \eqref{lem1eq3} for $s \neq 0$, and we see that $\partial_s R_a f(s,\theta^*)$ is bounded thus proving the first statement of the theorem. It remains to analyse the jump at $s = 0$.

Let us first consider the jump in $\partial_s Da((x,s),\theta^*)$ across $s = 0$. The only terms contributing to this jump in \eqref{lem3eq1} will be those with $\phi_i^\pm$ where $t_i^\pm(s) \rightarrow 0$ as $s \rightarrow 0^\pm$ since the others correspond to boundaries which do not have corners along $(s^*,\theta^*)$ and the line is not tangent to any of the boundaries. Let us reindex the indices $i$ corresponding to such $t_i$ using a new index $k$ as $\{ i^\pm_{k}\}_{k=1}^{\tilde{N}^\pm}$. Then the jump in $\partial_s Da((x,s),\theta^*)$ is given by
\[
\begin{split}
\Big [ \partial_s Da((x,s),\theta^*)\Big ]_{0^-}^{0^+} & = \lim_{s \rightarrow 0^+} \partial_s Da((x,s),\theta^*) - \lim_{s \rightarrow 0^-} \partial_s Da((x,s),\theta^*)\\
& = \left [ \sum_{k=1}^{\tilde{N}^+} (c^+_{i^+_k-1} -c^+_{i^+_k}) \partial_s \phi^+_{i_k^+}(x,0^+)\right ] - \left [ \sum_{k=1}^{\tilde{N}^-} (c^-_{i^-_k-1} -c^-_{i^-_k}) \partial_s \phi^-_{i_k^-}(x,0^-)\right ].
\end{split}
\]
Using \eqref{lem1eq3} we find that the jump of $\partial_s R_a f(s,\theta^*)$ across $s = 0$ will be
\[
\begin{split}
\Big[  \partial_s R_a f(s,\theta^*) \Big ]^{0^+}_{0^-} & = \lim_{s \rightarrow 0^+} \partial_s R_a f(s,\theta^*)  - \lim_{s \rightarrow 0^-} \partial_s R_a f(s,\theta^*)  \\
& = - \left (\left [ \sum_{k=1}^{\tilde{N}^+} (c^+_{i^+_k-1} -c^+_{i^+_k}) \partial_s t^+_{i_k^+}(0^+) \right ] - \left [ \sum_{k=1}^{\tilde{N}^-} (c^-_{i^-_k-1} -c^-_{i^-_k})  \partial_s t^-_{i_k^-}(0^-)\right ] \right )\\
& \hskip1cm \times \int_{-\infty}^0 f(x,0) e^{-Da((x,0),\theta^*)} \ \mathrm{d} x.
\end{split}
\]
Taking into account the rotation and translation used at the beginning we see that this corresponds with \eqref{eq:corner} and so completes the proof.
\end{proof}

\noindent Next we consider $\mathcal{K}_{a,1}$ which requires a bit more work.

%

\begin{lemma}\label{curveint}
Suppose that $f \in C_c^1(\mathbb{R}^2)$, $a$ is multi-bang, and that $(s^*,\theta^*) \in \mathcal{K}_{a,1}$ is such that the line corresponding to $(s^*,\theta^*)$ is only tangent to a boundary $\partial \Omega_j$ at one point given by $s^* (\theta^*)^\perp + t^* \theta^*$ which is not also a corner. Furthermore, suppose that $a = c$ on the convex side of the point of tangency, $a = c_0$ on the concave side, the curvature of $\partial \Omega_j$ at the point of tangency is $\kappa > 0$, and $\theta_\perp^* \in \mathbb{S}^1$ is orthogonal to $\theta^*$ and pointing into the convex side. Then
\begin{equation}\label{eq:tangent}
\begin{split}
\lim_{s\rightarrow (s^*)^{\pm}}|s-s^*|^{\frac{1}{2}}\partial_{s}R(s ,\theta^*)= \pm\frac{c_0-c}{\sqrt{\kappa/2}} \int_{-\infty}^{t^*}f(s^* (\theta^*)^\perp + t \theta^*))e^{-Da(s^* (\theta^*)^\perp + t \theta^*,\theta^*)}\ \mathrm{d}t
\end{split}
\end{equation}
where $\pm$ is the sign of $(\theta^*)^\perp \cdot \theta^*_\perp$.
\end{lemma}

\begin{proof}
After possibly rotating as in the proof of Lemma \ref{lem:cont} and reflecting across the $x$-axis we can assume that $s^* = 0$, $\theta^* = (1,0)$ and $\theta_\perp^* = (0,1)$. We further assume W.L.O.G.  by translating if necessary that the single point of tangency is the origin. This means that locally near the origin the boundary $\partial \Omega_j$ will be given as a graph in the form
\begin{equation}\label{boundary}
y = x^2 g(x)
\end{equation}
where $g$ is a strictly positive function and $g(0) = \frac{\kappa}{2}$.

For sufficiently small $s>0$ we then follow the same reasoning as in the proof of Lemma \ref{lem:cont} and so obtain (for $s>0$) the equation \eqref{lem1eq2}. The difference from Lemma \ref{lem:cont} is that in the current case the derivatives of the $t_i$'s corresponding to the point of tangency will blow up as $s \rightarrow 0^+$. There will be two such $t_i$'s which both go to $0$ as $s\rightarrow 0^+$, one positive and one negative which we label respectively as $t_\pm$, and differentiating \eqref{boundary} we can show that
\begin{equation}\label{lem2eq1}
\lim_{s\rightarrow 0^+} s^{1/2} \partial_s t_\pm(s) = \pm \frac{1}{\sqrt{2 \kappa}}.
\end{equation}
Now we have \eqref{lem1eq3} which holds for $s>0$ sufficiently small, and all the terms in this will be bounded as $s \rightarrow 0^+$ except those that involve derivatives of $t_\pm$. Thus when we multiply by $s^{1/2}$ and take the limit $s \rightarrow 0^+$ the only terms that will possibly remain are
\[
\begin{split}
\lim_{s\rightarrow 0^+} s^{1/2} \partial_{s}R_{a}f(s,\theta^*)& = \lim_{s \rightarrow 0^+} \Bigg\{ - s^{1/2}(c-c_0) \partial_s t_+(s) \int_{t_-(s)}^{t_+(s)} f(x,s) e^{-Da((x,s),\theta^*)} \mathrm{d} x\\
& \quad  -s^{1/2} (c - c_0) \left (\partial_s t_+(s) - \partial_s t_-(s)  \right )  \int_{-\infty}^{t_-(s)}  f(x,s) e^{-Da((x,s),\theta^*)} \mathrm{d} x  \Bigg\}.
\end{split}
\]
The term on the right on the first line is zero by \eqref{lem2eq1} and using the fact the integrand is continuous, while using \eqref{lem2eq1} another time we can evaluate the term on the second line to get
\[
\lim_{s\rightarrow 0^+} s^{1/2} \partial_{s}R_{a}f(s,\theta^*) = \frac{c_0-c}{\sqrt{\kappa/2}} \int_{-\infty}^{0}  f(x,0) e^{-Da((x,0),\theta^*)} \mathrm{d} x.
\]
Taking into account the rotation, translation and reflection at the beginning of the proof, this corresponds with \eqref{eq:tangent} and so completes the proof.
		\end{proof}
\begin{remark}
Note that Lemma \ref{curveint} precisely characterises the leading order singularity of the derivative $\partial_s R_a f$ at $(s^*,\theta^*)$. It is possible to obtain a similar formula and characterisation for some smooth parts of the boundary where the curvature is zero with some higher order derivative which does not vanish at $x=0$. In this case we would use $y=x^{2n}g(x)$ in place of \eqref{boundary} where $n\geq 2$ and $g(0)\neq0$. The order of the singularity in $\partial_s R_af$ as $s\rightarrow (s^*)^{\pm}$ is then $1-\frac{1}{2n}$ rather than $1/2$. 
\end{remark}

\begin{remark}
It is possible to combine the methods of proof of the previous lemmas to characterise the singularities at $(s^*,\theta^*)$ corresponding to lines both tangent to the boundaries at multiple places and/or passing through through multiple corners, but to simplify the statements we have not done this explicitly.
\end{remark}
		
\noindent At this point we note that these first three lemmas already show how we can determine some information about multi-bang $a$ from $R_af$. First of all, Lemma \ref{lem:cont} shows that if $R_af(s,\theta)$ is not continuous in $s$ at a point $(s^*,\theta^*)$, then the corresponding line must either be tangent to or passing through a corner of the boundary of one of the regions $\Omega_j$. If $\partial_s R_a f(s,\theta)$ is bounded near $(s^*,\theta^*)$, but has a jump in $s$ at this point, then the line must be passing through a corner from Lemma \ref{lem:corner}. If $\partial_s R_a f(s,\theta^*)$ blows up as $s \rightarrow s^*$, then the line $(s^*,\theta^*)$ must be tangent to one of the boundaries by Lemma \ref{curveint}. This already gives most of the information required to prove Theorem \ref{thm1} except for the part about the derivative with respect to $\omega$. For this we include one additional lemma studying the derivative with respect to variation in the angle $\omega$ rather than $s$ as in the previous lemmas.

\begin{lemma}\label{lem:omega}
Assume the same hypotheses as in Lemma \ref{curveint}. Additionally suppose that\\ $\theta^* = (\cos(\omega^*),\sin(\omega^*))$ and $x^*$ is a point on the line corresponding to $(s^*,\theta^*)$. Then
\begin{equation}\label{eq:omega}
\begin{split}
&\lim_{\omega\rightarrow (\omega^*)^{s_1}}|\sin(\omega-\omega^*)|^{1/2} \partial_{\omega}R(x^*\cdot \theta^\perp,\theta)=\\
& \hskip2cm s_2(c_0-c) \sqrt{\frac{2|x^* \cdot \theta^* -t^*|}{\kappa}} \int_{-\infty}^{t^*}f(s^* (\theta^*)^\perp + t \theta^*))e^{-Da(s^* (\theta^*)^\perp + t \theta^*,\theta^*)}\ \mathrm{d}t
\end{split}
\end{equation}
where $s_1$ is the sign of $(t^* - x^* \cdot \theta^*) (\theta^*)^\perp \cdot \theta^*_\perp$, and $s_2$ is the sign of $(\theta^*)^\perp \cdot \theta^*_\perp$. In the case that $t^* - x^* \cdot \theta^* = 0$, the equation \eqref{eq:omega} still holds with $s_1$ removed (note the right hand side is zero in that case).
\end{lemma}	

\begin{proof}
As in the previous lemmas by rotating, translating and possibly reflecting about the x-axis we assume W.L.O.G. that $\theta^* = (1,0)$, $\theta^*_\perp = (0,1)$ and the point of tangency is at the origin (i.e. $t^* = 0$). We also assume that $x^* = (\ell,0)$ and for the moment consider only the case $\ell \neq 0$. Note that the line corresponding to $(x^* \cdot \theta^\perp,\theta)$ is precisely the line through $x^*$ tangent to $\theta$, and we will change the parametrisation of this line in the integral definition of the AtRT so that $t = 0$ always corresponds with $x^*$. After doing all of this we have
\begin{equation}\label{Rafomega}
R_a f(x^* \cdot \theta^\perp,\theta) = \int_{-\infty}^\infty f(x^* + t \theta) e^{-Da(x^* + t \theta,\theta)} \ \mathrm{d} t.
\end{equation}
Now we use the same notation as in the previous lemmas and label the ordered values of $t$ along the line $t \mapsto x^* + t \theta$, for $\mathrm{sgn}(\ell) \omega > 0$ and $|\omega|$ sufficiently small, at which the line intersects one of the boundaries $\partial \Omega_j$ as $\{t_i(\omega)\}_{i=1}^N$. Two of the $t_i$ will correspond to the point of tangency and these will satisfy $t_i(\omega) \rightarrow -\ell$ as $\omega \rightarrow 0^{-\mathrm{sgn}(\ell)}$. Combining \eqref{boundary} with the geometric relations
\begin{equation}\label{eq:ell0}
\cos(\omega) = \frac{x-\ell}{t_\pm}, \quad \sin(\omega) = \frac{y}{t_\pm}, \quad \tan(\omega)= \frac{y}{x-\ell}
\end{equation}
we can show by taking derivatives with respect to $\omega$, and some computation, that
\begin{equation}\label{lem4eq1}
\lim_{\omega \rightarrow 0^{-\mathrm{sgn}(\ell)}} |\sin(\omega)|^{1/2} \partial_\omega t_\pm(\omega) = \pm \sqrt{\frac{|\ell|}{2\kappa}}.
\end{equation}
In the case that $\ell = 0$ we will still have $t_\pm(\omega)$ when $\omega \neq 0$ is sufficiently small corresponding to the two intersections near the tangent point, but $t_{-\mathrm{sgn}(\omega)}(\omega) = 0$ and $\omega\ t_{\mathrm{sgn}(\omega)}(\omega) > 0$ for all $\omega \neq 0$. In this case we can show in a similar manner to the $\ell \neq 0$ case that
\begin{equation}\label{ell0lim}
\lim_{\omega \rightarrow 0} |\sin(\omega)|^{1/2} \partial_\omega t_\pm(\omega)  = 0.
\end{equation}
We next define functions $\phi_i$ in a similar way to before (compare with \eqref{lem1phi}) as
\begin{equation}\label{lem4phi}
\phi_i(t,\omega) = \left \{
\begin{array}{ll}
t_i(\omega) & t < t_i(\omega)\\
t & t \geq t_i(\omega),
\end{array}
\right .
\end{equation}
for $-\mathrm{sgn}(\ell) \omega > 0$. As before
\[
Da(x^*+t\theta,\theta) = \sum_{i=1}^{N} (c_{i-1}-c_i) \phi_i(t,\omega)
\]
also for $-\mathrm{sgn}(\ell) \omega > 0$. In the case $\ell = 0$ we still have versions of the previous formula for $\omega \neq 0$, but it will change depending on the sign of $\omega$. We will also write $\phi_\pm$ for those $\phi_\pm$ corresponding to $t_\pm$.

Now let us take the derivative of \eqref{Rafomega} in the case when  $-\mathrm{sgn}(\ell) \omega > 0$ if $\ell \neq 0$ or $\omega \neq 0$ if $\ell = 0$. We then have
\begin{equation}\label{domegaRaf}
\partial_\omega R_a f(x^* \cdot \theta^\perp,\theta) = \int_{-\infty}^\infty \Big ( \partial _\omega f(x^* + t \theta) - \partial_\omega Da(x^* + t\theta, \theta) f(x^* + t \theta) \Big ) e^{-Da(x^* + t \theta,\theta)} \ \mathrm{d} t.
\end{equation}
First consider the case $\ell = 0$. In this case when we multiply by $|\sin(\omega)^{1/2}|$ and take the limit as $\omega \rightarrow 0$, using \eqref{ell0lim} we see that the limit is zero. Since $\ell = x^* \cdot \theta^* - t^*$, this proves the result when $\ell = 0$. Now consider when $\ell \neq 0$. In this case we multiply by $|\sin(\omega)|^{1/2}$ and take the limit as $\omega \rightarrow 0^{-\mathrm{sgn}(\ell)}$. The only terms that are not bounded in \eqref{domegaRaf} for $\omega$ close to zero are those that involve derivatives of $\phi_\pm$. We therefore have
\[
\begin{split}
\lim_{\omega \rightarrow 0^{-\mathrm{sgn}(\ell)} }|\sin(\omega)|^{1/2} \partial_\omega R_a f(x^* \cdot \theta^\perp,\theta)& = \\
&\hskip-3cm (c-c_0) \lim_{\omega \rightarrow 0^{-\mathrm{sgn}(\ell)} } \int_{-\infty}^{t_-(\omega)} |\sin(\omega)|^{1/2} \partial_\omega t_-(\omega) f(x^* + t \theta) e^{-Da(x^* + t \theta,\theta)} \ \mathrm{d} t\\
&\hskip-2cm +(c_0-c)\lim_{\omega \rightarrow 0^{-\mathrm{sgn}(\ell)} } \int_{-\infty}^{t_+(\omega)} |\sin(\omega)|^{1/2} \partial_\omega t_+(\omega) f(x^* + t \theta) e^{-Da(x^* + t \theta,\theta)} \ \mathrm{d} t.
\end{split}
\]
Applying \eqref{lem4eq1} to this we finally obtain
\[
\lim_{\omega \rightarrow 0^{-\mathrm{sgn}(\ell)} }|\sin(\omega)|^{1/2} \partial_\omega R_a f(x^* \cdot \theta^\perp,\theta) = (c_0 - c) \sqrt{\frac{2 |\ell|}{\kappa}} \int_{-\infty}^{-\ell}  f(x^* + t \theta^*) e^{-Da(x^* + t \theta^*,\theta^*)} \ \mathrm{d} t
\]
Taking into account the translations, rotation and reflection from the beginning of the proof this formula agrees with \eqref{eq:omega}, and so completes the proof.
\end{proof}

Before giving the proof of Theorem \ref{thm1} we record a corollary of the proof of Lemma \ref{lem:omega} which looks at one case in which at the point of tangency of a line to the boundary of an $\Omega_j$, the curvature of the boundary is zero. This corollary will be useful for the proof of Theorem \ref{thm2} later.

\begin{corollary}\label{cor:0curv}
Assume the same hypotheses as in Lemma \ref{lem:omega} including the assumption that there is a convex and concave side of the boundary near the point of tangency, but say the curvature is $\kappa = 0$ at the point of tangency. If $x^* \cdot \theta^* - t^* \neq 0$ and the ray $\{x^* + t \theta^* \ : \ t < t^*\}$ intersects the set $\{f >0\}$, then limits in \eqref{eq:omega} and \eqref{eq:tangent} are one of $\pm \infty$.
\end{corollary}

\begin{proof}
The proof follows the same outline as the proofs of Lemma \ref{curveint} and \ref{lem:omega}, but in \eqref{boundary} we have $g(0) = 0$. Because of this \eqref{lem2eq1} and \eqref{lem4eq1} respectively change to
\[
\lim_{s \rightarrow 0^+} s^{1/2} \partial_s t_\pm(s) = \pm\infty, \quad \mbox{and} \quad \lim_{\omega \rightarrow 0^{-\mathrm{sgn}(\ell)}} |\sin(\omega)|^{1/2} \partial_\omega t_\pm(\omega) = \pm \infty.
\]
Following the proofs through the rest of the way with this change, and using the fact that the integrals appearing at the end do not vanish, proves the corollary.
\end{proof}

\noindent The proof of Theorem \ref{thm1} now follows simply from Lemmas \ref{lem:corner}, \ref{curveint} and \ref{lem:omega} as we now point out.\\

\noindent{{\it Proof of Theorem \ref{thm1}.}}
For the first item in Theorem \ref{thm1} we note that if the hypotheses are satisfied, then by Lemma \ref{curveint} equation \eqref{eq:tangent} holds, and since the ray $\{x + t \theta^* \ | \ t<0\}$ intersects $\{f>0\}$, the integral on the right side of \eqref{eq:tangent} is not zero. Therefore $\partial_s R_a f(s,\theta^*)$ has a singularity of order $1/2$ at $s = s^*$. Similarly Lemma \ref{lem:omega} implies that \eqref{eq:omega} will hold and this limit will only be zero when $x^* = x$. This proves the first part.

The second part follows similarly from Lemma \ref{lem:corner}, although we note the under the given hypotheses $N=2$ in \eqref{lem:corner} and $b_1 = -b_2 \neq 0$. Thus we have a jump in $\partial_s R_a f$ if
\[
\tan(\alpha_1) \neq \tan(\alpha_2).
\]
This will always be true at a corner since there would be equality only if $\alpha_1 = \alpha_2 + n \pi$ for an integer $n$, but that would mean we are not at a corner.
\qed\\

\noindent To finish this section we prove one more lemma concerning what can happen if there is a flat section of a boundary of one of the $\Omega_j$. This will be used in the proof of Theorem \ref{thm2}.

\begin{lemma}\label{lem:flat}
Assume that $a$ is multi-bang and $f \in C^1_c(\mathbb{R}^2)$ is non-negative. Suppose that the line given by $(s^*,\theta^*)$ intersects the boundary of one of the $\Omega_j$ in a line segment of length $\ell$ given by $\{s^* (\theta^*)^\perp + t \theta^* \ | \ t \in [t_- ,t_+]\}$, and that there are no corners for any of the other regions contained in the interior of this line segment. Assume also that the ray $\{s^* (\theta^*)^\perp + t \theta^* \ | \ t < t_+\}$ intersects the set $\{f>0\}$. Then $R_af(s,\theta^*)$ is discontinuous at $s = s^*$.
\end{lemma}

\begin{proof}
We follow the same method as the proofs of Lemma \ref{lem:corner} using the notation $t_i^\pm$ and $\phi^\pm_i$ as before. The difference here is that some of these $t_i^\pm(\omega)$ will converge to the endpoints $t_-$ and $t_+$ of the line segment as $s \rightarrow 0^\pm$. These will lead to a jump in $Da((x,s),\theta^*)$ given by \eqref{lem1eq2} at $s = 0$ when $x<t_+$. This jump will then lead to a jump in $R_af$ if $f$ satisfies the given hypothesis.
\end{proof}

\noindent We next proceed to the statement and proof of Theorem \ref{thm2} as well as some related results.

\subsection{Theorem 2 and related results} \label{sec:thm2}

In this section we state and prove Theorem \ref{thm2}.

\begin{theorem}\label{thm2}
Suppose that $a$ is nicely multi-bang (see Definition \ref{def:nicemb}) and $f \in C^1_c(\mathbb{R}^2)$ is non-negative. Also assume that
\begin{enumerate}
\item for all $x \in \mathcal{P}_{a,1}$ the line tangent to a boundary at $x$ passes through the set $\{f>0\}$, and
\item for all $x \in \mathcal{P}_{a,2}$ there is a line passing through $x$ that also passes through the set $\{f>0\}$.
\end{enumerate}
Then $a$ and $f$ are uniquely determined by $R_af$.
\end{theorem}

\noindent We prove Theorem \ref{thm2} throughout this section in a series of lemmas. The initial step in the proof of Theorem \ref{thm2} is to show that under the given hypotheses we can determine the set of points where $a$ jumps. We will do this now.

\begin{lemma}\label{lem:boundary}
Assume the same hypotheses as Theorem \ref{thm2}. Then we can determine from $R_af$ the sets $C_j$ appearing in Definition \ref{def:nicemb} for $a$.
\end{lemma}

\begin{proof}
We first note that by Lemmas \ref{lem:cont} and \ref{lem:corner} and Corollary \ref{cor:0curv} the set of $(s^*,\theta^*)$ such that $\partial_s R_a f(s,\theta^*)$ for $s$ near $s^*$ is not bounded gives the set of lines which are tangent to some boundary $\partial C_j$, possibly missing some of the lines which intersect a boundary in a line segment. We can get rid of all of the $(s^*,\theta^*)$ corresponding to lines which intersect a boundary $\partial C_j$ in a line segment by looking at the continuity of $R_af(s,\theta^*)$ near $s^*$ and using Lemma \ref{lem:flat}. Thus we can determine the set of $(s^*,\theta^*)$ such that the corresponding lines are tangent to a boundary $\partial C_j$ at some point, and since the $C_j$ are nested convex sets the point of tangency along each such line must be unique. We can determine the point of tangency along each line that is tangent at a point where the curvature of $\partial C_j$ is not zero using Theorem \ref{thm1} or we can determine if at the point of tangency the curvature is zero using Corollary \ref{cor:0curv}. Thus we can identify all points in the boundaries of the $C_j$ at which the curvature of $\partial C_j$ is not zero. Next we will show that we can also find the corners of the boundaries $\partial C_j$.

By Lemma \ref{lem:corner} and the hypotheses, for every corner point $x$ for some $\partial C_j$ there will infinitely many lines passing through $x$ such that for at least $(s^*,\theta^*)$ corresponding to these lines $\partial_s R_a f(s,\theta^*)$ is bounded, but has a jump at $s = s^*$. This allows us to determine the corner points, and combining this with the previous paragraph we see that we can determine the $\mathcal{P}_a$ from $R_af$ under the given hypotheses. We next show that this is sufficient to determine all of the $C_j$.

By Lemma \ref{lem:cohull}, which we will prove next, the closure of the convex hull of $\mathcal{P}_a$ is equal to the closure of $C_1$. Therefore we can determine $C_1$. The rest of the sets $C_j$ can now be determined inductively. Indeed, suppose that we know $C_{l}$ for all  $l<j$. Then by Lemma \ref{lem:cohull} again
\[
\overline{C_j} = \overline{\mathrm{conhull}\left ( \mathcal{P}_a \setminus \bigcup_{l=1}^{j-1}\partial C_l \right )},  
\]
and so we can determine $C_j$. This completes the proof. 
\end{proof}

\noindent The following geometric lemma was needed in the proof of Lemma \ref{lem:boundary}.

\begin{lemma} \label{lem:cohull}
Suppose that $C \subset \mathbb{R}^2$ is closed, convex, bounded and has smooth boundary possibly with corners. Also let $\mathcal{P}$ be the subset of $\partial C$ consisting of points which are either corners of $\partial C$, or where $\partial C$ has nonzero curvature. Then
\[
C = \overline{\mathrm{conhull} \left ( \mathcal{P} \right )}
\]
where $\mathrm{conhull}(\mathcal{P})$ is the convex hull of $\mathcal{P}$.
\end{lemma}
\begin{proof}
Since $C$ is closed and convex, and $\mathcal{P} \subset C$, we have $\overline{\mathrm{conhull} \left ( \mathcal{P} \right )} \subset C$. Thus it only remains to show the opposite inclusion. Suppose that $x \in \partial C \setminus \overline{\mathrm{conhull} \left ( \mathcal{P} \right )}$. Then there must be a neighbourhood $U$ of $x$ such that $U \cap \partial C$ does not intersect $\mathcal{P}$. Therefore the curvature of $\partial C$ is zero at all points in $U \cap \partial C$, and since $x$ is also not a corner for $\partial C$, this implies that $U \cap \partial C$ must contain a line segment containing $x$ in its relative interior. There must then be some maximally extended line segment containing $x$ which is contained in $\partial C$. At least one of the end points of this maximal line segment must also be in  $\partial C \setminus \overline{\mathrm{conhull} \left ( \mathcal{P} \right )}$ since otherwise we would have $x \in \overline{\mathrm{conhull} \left ( \mathcal{P} \right )}$ by convexity. However this is a contradiction since by the argument we have already given this endpoint would be in the relative interior of a line segment contained in $\partial C \setminus \overline{\mathrm{conhull} \left ( \mathcal{P} \right )}$. Thus $\partial C \setminus \overline{\mathrm{conhull} \left ( \mathcal{P} \right )} = \emptyset$, which then implies the result.
\end{proof}

\noindent Having proven in Lemma \ref{lem:boundary} that we can determine the sets $C_j$ from $R_a f$ under the hypotheses of Theorem \ref{thm2}, it remains to show that we can recover $f$ and the jumps in $a$ across each of the boundaries. For this we argue by induction starting at the outermost region $C_1$, and continuing inward. First suppose that $a$ and $f$ are known everywhere outside of $C_{j-1}$ for some $j \geq 2$. Then since there must be at least one point $x$ on the boundary $\partial C_{j-1}$ in $\mathcal{P}_a$, we can either use \eqref{eq:corner} or \eqref{eq:tangent} to determine the jump in $a$ across the boundary at $x$. Therefore we can determine $a$ outside of $C_j$. To complete the induction step it then remains to show we can determine $f$ outside of $C_j$. For this we use the following lemma.

\begin{lemma}\label{lem:f}
Suppose that $f \in C_c^1(\mathbb{R}^2)$, and a is nicely multi-bang with sets $\{C_j\}_{j=1}^n$, and let $C_0$ be an open ball centred at the origin that is sufficiently large so that $C_1 \Subset C_0$ and $\mathrm{supp}(f) \Subset C_0$. Then for $j \geq 1$, $f|_{C_{j-1} \setminus C_j}$ is uniquely determined if we know all of
\begin{enumerate}
\item $R_a f$,
\item the sets $C_j$,
\item $a|_{\mathbb{R}^2 \setminus C_{j}}$, and
\item $f|_{\mathbb{R}^2 \setminus C_{j-1}}$.
\end{enumerate}
\end{lemma}

\begin{proof}
For this proof we will write $(x,y)$ as Cartesian coordinates for points in $\mathbb{R}^2$. By translating and rotating as necessary we assume W.LO.G. that $C_{j}$ is contained in the upper half plane $\{y>0\}$, and show that we can then uniquely determine $f$ restricted to the lower half-plane $\{y<0\}$. By translating to bring $C_{j}$ arbitrarily close to $\{y = 0\}$ and rotating this then shows we can determine $f$ everywhere outside of $C_{j}$ and so will complete the proof. 

Having done the transformations described in the previous paragraph, we also assume that $C_{j-1} \subset \{y > -h\}$. Now choose $\omega > \epsilon > 0$ such that the parabola $\{ y= \epsilon x^2\}$ lies entirely outside of $C_{j}$ and the parabola $\{ y= \omega x^2 - h\}$ lies entirely outside of $C_{j-1}$. It is possible to find such $\omega$ and $\epsilon$ since the $C_j$ are all bounded. Now choose $\phi \in C^\infty(\mathbb{R}^2)$ such that $\phi(x,y) = 1$ on $C_{j-1}$ and $\phi(x,y) = 0$ on the set $\{y < \omega x^2 - h\}$. Then define $\tilde{f} = \phi f$ so that $\tilde{f}$ has support contained in the set $\{y \geq \omega x^2 - h\}$ and is such that $\tilde{f}|_{C_{j-1}} = f|_{C_{j-1}}$. Also, supposing that $a = c$ on $C_{j-1} \setminus C_j$, we set
\[
\tilde{a} = \left \{
\begin{array}{ll}
\tilde{a}(x,y) = a(x,y) & (x,y) \in C_{j-1}\\
\tilde{a}(x,y) = c & (x,y) \in (\mathbb{R}^2 \setminus C_{j-1})\cap \{ y>\epsilon x^2 - h-1\}\\
\tilde{a}(x,y) = 0 & \mbox{otherwise}.
\end{array}
\right .
\]
The setup described in the last few lines is illustrated in figure \ref{parabolasetup}. Our next step is to show that we can determine $R_{\tilde{a}} \tilde{f}(s,\theta)$  if $(s,\theta)$ corresponds to a line contained in the set $\{y< \epsilon x^2\}$ given the hypotheses of the lemma. If this line does not pass through $C_{j-1}$, then there is no problem since we know $\tilde{f}$ and $\tilde{a}$ outside of $C_{j-1}$. Suppose on the other hand that the line does pass through $C_{j-1}$, and let the two points of intersection between the line and $\partial C_{j-1}$ be denoted $t_1<t_2$ (note there will always be two such points by convexity and these can be determined from $C_j$).  We then have
\[
\begin{split}
R_af(s,\theta) & = \int_{-\infty}^{t_1} f(s\theta_\perp + t \theta) e^{-Da(s \theta_\perp + t \theta,\theta)} \ \mathrm{d} t + \int_{t_1}^{t_2} f(s\theta_\perp + t \theta) e^{-Da(s \theta_\perp + t \theta,\theta)} \ \mathrm{d} t \\
& \hskip1cm + \int_{t_2}^{\infty} f(s\theta_\perp + t \theta) e^{-Da(s \theta_\perp + t \theta,\theta)} \ \mathrm{d} t.
\end{split}
\]
The first and third terms on the right side of the last equation only involve $a|_{\mathbb{R}^2 \setminus C_{j}}$ and $f|_{\mathbb{R}^2 \setminus C_{j-1}}$ as well as $t_1$ and $t_2$, and thus are known functions of $(s,\theta)$ under the given hypotheses. We combine these together, and also $R_a f$, into one function $G(s,\theta)$, and so, since also $f |_{C_{j-1}} = \tilde{f} |_{C_{j-1}}$, we have
\[
 \int_{t_1}^{t_2} \tilde{f}(s\theta_\perp + t \theta) e^{-Da(s \theta_\perp + t \theta,\theta)} \ \mathrm{d} t = G(s,\theta)
\]
where $G$ is a function which can be determined from the known information. Next note that for $t \in (t_1,t_2)$, $F = -Da(s \theta_\perp + t \theta,\theta) + D\tilde{a}(s \theta_\perp + t \theta,\theta)$ only depends on $s$ and $\theta$, and can be determined under the hypotheses. Therefore we have
\[
 \int_{t_1}^{t_2} \tilde{f}(s\theta_\perp + t \theta) e^{-D\tilde{a}(s \theta_\perp + t \theta,\theta)} \ \mathrm{d} t = e^{-F(s,\theta)} G(s,\theta).
\]
Finally, we can add back in the integrals with $\tilde{f}$ and $\tilde{a}$ from $-\infty$ to $t_1$ and $t_2$ to $\infty$ since these only involve $\tilde{a}|_{\mathbb{R}^2 \setminus C_{j}}$, and $\tilde{f}|_{\mathbb{R}^2 \setminus C_{j-1}}$ which are assumed to be known. Doing this we see that $R_{\tilde{a}} \tilde{f}$ can be determined given the hypotheses. The problem has now been reduced to determining $\tilde{f} |_{C_{j-1} \cap \{y<0\}}$.

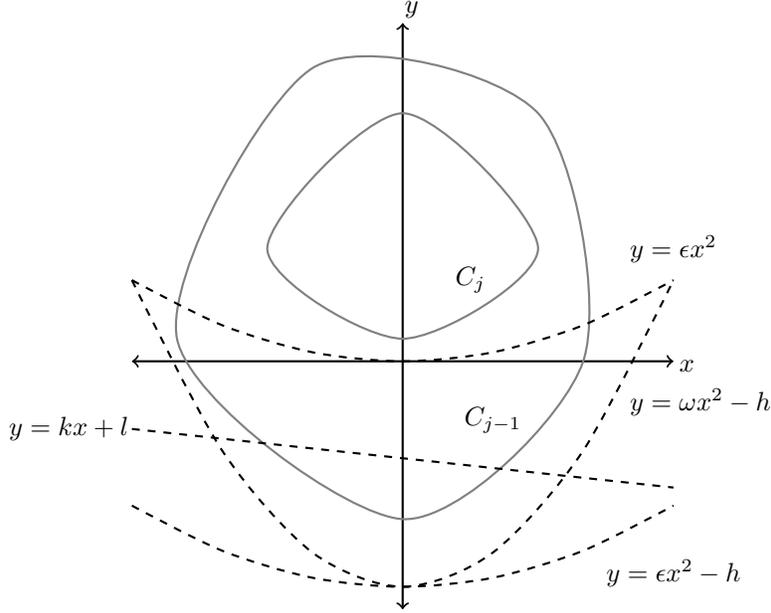
\begin{figure}
\begin{center}
\begin{tikzpicture}[scale=.6]
\draw[black,thick,<->](-1,10.5)--(11,10.5);
\draw[black,thick,<->](5,5)--(5,18);
\node at (11.3,10.4){$x$};
\node at (5.2,18.3){$y$};
\draw [gray,thick] plot [smooth cycle] coordinates{(2,13) (5,11) (8,13) (5,16)};
\draw[gray,thick] plot [smooth cycle] coordinates{(0,11) (5,7)(9,10.5)(8,16)(3,17)};
\draw[black,thick,dashed] plot[smooth] coordinates {(-1,12.3)(1,11.3)(3,10.7)(5,10.5)(7,10.7)(9,11.3)(11,12.3)};
\draw[black,thick,dashed] plot[smooth] coordinates {(-1,7.3)(1,6.3)(3,5.7)(5,5.5)(7,5.7)(9,6.3)(11,7.3)};
\node at (11,5.8){$y=\epsilon x^{2}-h$};
\node at (11,13){$y=\epsilon x^{2}$};
\draw[black,thick,dashed] plot[smooth] coordinates{(-1,12.3)(1,8.52)(3,6.26)(5,5.5)(7,6.26)(9,8.52)(11,12.3)};
\node at (11.6,9.6){$y=\omega x^{2}-h$};
\node at (7,9.2){$C_{j-1}$};
\node at (6.5,12.3){$C_j$};
\draw[thick,black,dashed] plot coordinates{(-1,9)(11,7.7)};
\node at (-2.4,9){$y = kx + l$};
\end{tikzpicture}
\caption{This illustrates the setup and some of the notation used in the proof of Lemma \ref{lem:f}. Note that we assume $a = c$ in the region $C_{j-1} \setminus C_j$, and $\tilde{a} = c$ in the region above the lowest parabola translated downwards by $1$ and outside of $C_j$. $f$ is assumed to be known outside of $C_{j-1}$, and then $\tilde{f}$ is supported in the region above the middle parabola.}
\label{parabolasetup}
\end{center}
\end{figure}

Our final step is to change variables in order to reduce the problem to the one considered in \cite[Theorem 3.1]{Bukgeim}. For this we consider only $R_{\tilde{a}}\tilde{f}(s,\theta)$ for $(s, \theta)$ corresponding to lines contained in $\{y< \epsilon x^2\}$. We reparametrise such lines using $y = kx + l$ where the slope $k$ and intercept $l$ replace $\theta$ and $s$ respectively. We will also write  $x_+(k,l) = \frac{k + \sqrt{k^2 + 4\epsilon (l+h+1)}}{2\epsilon}$ for the larger value of $x$ at which $\{y = kx + l\}$ intersects $\{ y = \epsilon x^2 - h-1\}$. When we parametrise the lines in this way, the beam transform becomes
\[
D\tilde{a}((x,kx+l),k) = \int_0^\infty \tilde{a}(x+s,k(x+s) + l)\sqrt{1+k^2} \ \mathrm{d} s = c(x_+(k,l) - x) \sqrt{1+k^2}
\]
and so the AtRT becomes
\[
R_{\tilde{a}} \tilde{f}(k,l) = \int_{-\infty}^\infty \tilde{f}(x,kx+l) e^{c (x - x_+(k,l))\sqrt{1+k^2}} \sqrt{1+k^2} \ \mathrm{d} x.
\]
We now introduce the new coordinates $(z,w)$ defined by $z = \sqrt{\epsilon} x$ and $w = y - \epsilon x^2 + h$. With this change, the region $\{\epsilon x^2 > y > \epsilon x^2 - h\}$ becomes the strip $\{h>w> 0\}$, and abusing notation slightly by writing $\tilde{f}$ also for the same function in these coordinates the AtRT becomes
\[
R_{\tilde{a}} \tilde{f}(k,l) = \int_{-\infty}^\infty \tilde{f}\left (z,-\left (z-\frac{k}{2 \sqrt{\epsilon}}\right )^2 + \frac{k^2}{4 \epsilon} + l - h \right ) e^{c (\sqrt{\epsilon} z - x_+(k,l))\sqrt{1+k^2}} \sqrt{\epsilon (1+k^2)} \ \mathrm{d} z
\]
for any $(k,l)$ corresponding to a line contained in $\{y<\epsilon x^2\}$ and passing through the region $\{ y > \epsilon x^2 - h\}$. The uniqueness of $\tilde{f}$ in the region $\{y<\epsilon x^2\}$, and therefore also $f$ in the same region, now follows from \cite[Theorem 3.1]{Bukgeim} since we have that
\[
a = e^{c (\sqrt{\epsilon} z - x_+(k,l))\sqrt{1+k^2}} \sqrt{\epsilon (1+k^2)}
\]
is an analytic function of $k/(2\sqrt{\epsilon})$ and $z$ provided the imaginary part of $k$ is sufficiently small.
\end{proof}

\noindent Lemma \ref{lem:f} now allows us to complete the proof of Theorem \ref{thm2}. Indeed, the induction step is already proved as described just above the lemma. The base case is also included in Lemma \ref{lem:f} since we can determine the set $C_0$ as in the lemma by looking at the support of $R_af$, and then we always know $f|_{\mathbb{R}^2 \setminus C_0} = 0$ and $a|_{\mathbb{R}^2 \setminus C_1} = 0$. This completes the proof of Theorem \ref{thm2}.

\section{Numerical method} \label{sec:nummethod}
				    			    
We now turn our attention to numerically recovering $a$ and $f$ from data. We begin by first outlining how we discretize the domain. Let $\Omega$ be the domain of interest, and split $\Omega$ into $M^{2}$ square pixels of resolution $\mathrm{d}x$. We order the pixels lexicographically from the top left to the bottom right. We then assume that $a$ and $f$ are piecewise constant over each pixel. Recall that for an oriented line given by $(s,\theta)$ we define the AtRT via
				    			    
\begin{equation}\label{discretepart}
R_{a}f(s,\theta)=\int_{-\infty}^{\infty}\!f(s\theta^{\perp}+t\theta)e^{-Da(s\theta^{\perp}+t\theta,\theta)}\,\mathrm{d}t
\end{equation}
where $s$ is the signed closest approach to the origin and $\theta$ is a unit direction tangent to the line and giving the orientation. Since $a$ and $f$ are piecewise constant on the pixels, we can evaluate \eqref{discretepart} exactly as follows. Let $P$ be a list of the pixels passed, in the order in which they are passed, along the oriented line and denote the length of $P$ by $N$. Note for this to be well-defined we need the ray to be oriented.  Let $K$ be the ordered set of $t$ values which correspond to an intersection with an edge of a pixel in the grid and let $IT$ be the set of distances between adjacent entries in $K$. Using this notation we find 
				    			    
\begin{equation}\label{DAtRT}
R_af(s,\theta)=\sum_{i=1}^{N}f_{P(i)}IT(i)e^{-\frac{IT(i)a_{P(i)}}{2}}\mathrm{sinhc}\left(\frac{IT(i)a_{P(i)}}{2}\right)S(i).
\end{equation}
where $a_{P(i)},f_{P(i)}$ are the values of $a$ and $f$ in the $P(i)$th pixel,				    			    
\begin{equation}\label{di}
\mathrm{sinhc}(z)=\begin{cases}				    			    
\hfill \frac{\mathrm{sinh}({z})}{{z}}\hfill &  { z \neq 0}\\				    			    
\hfill 1 \hfill & { z=0}\\
\end{cases}
\end{equation}
and $S(N)=1$, $S(i-1)=S(i)e^{-IT(i)a_{P(i)}}$. This allows us to rewrite the AtRT as a vector equation involving $a$ and $f$. If we are given data vector $d$ for a set $\mathcal{I}$ of oriented lines $(s_{i},\theta_{i})_{i\in\mathcal{I}}$ then we can combine all of these vector equations into a matrix equation 
\begin{equation}
\label{eq:discrete} R[a]f=d.
\end{equation}
The discretised problem of interest is then to determine both $a$ and $f$ from $d$ given by \eqref{eq:discrete} where $a$ is multi-bang with the admissible set $\mathcal{A} = \{a_0,a_1,...,a_n\}$ known (note that for notational convenience we have reindexed the admissible values relative to Definition \ref{def:multibang}). We attempt to do this by solving the variational problem
\begin{equation}\label{variationalproblem}				    			    
\mathrm{argmin}_{a,f}\mathcal{R}(a,f):=\|R[a]f-d\|^{2}+\alpha\mathcal{M}(a)+\lambda\TV(a)+\eta\TV(f)			    			    
\end{equation}
where $\TV$ is a discrete version of the total variation and $\mathcal{M}$, which will be described below, is used to enforce the multi-bang assumption. The known set of admissible attenuation values is $A:=\{a_0,a_1,...,a_n\}$ with $a_{0}<a_1<...<a_n$. Recent work in \cite{MB,MBorig} attempted to design a convex regularizer to promote multi-bang solutions when the admissible set is known. The original idea in \cite{MB} was to make a convex penalty with jumps in gradient at admissable values. In this paper we instead use a modified, non-convex, version of the multi-bang penalty given by
\begin{equation}\label{globalMB}
\mathcal{M}(a):=\int_{\Omega}m(a(x))\mathrm{d}x
\end{equation}
where
\begin{equation}\label{multibang}
m(t)=
\left \{
\begin{array}{cl}
(a_{i+1}-t)(t-a_{i}), & t\in[a_{i},a_{i+1}]\\
\infty, &\mathrm{otherwise}.
\end{array}
\right .
\end{equation}
Compared to the convex multi-bang penalty from \cite{MB}, this has the advantage of giving a proximal map which has multi-bang values as stationary points.
Note that in the discrete case we consider piecewise constant $a$ and so \eqref{globalMB} is really a sum over pixels given by
\begin{equation}\label{discreteMB}
\mathcal{M}(a):=\sum_{i=1}^{M^{2}}m(a(i)).				    			    
\end{equation}
Strictly speaking \eqref{discreteMB} should have a factor of $\mathrm{d}x^2$ in front of the summation but this gets absorbed by the regularization parameter $\alpha$ and so we omit it.

Although the regularizer \eqref{discreteMB} promotes multi-bang solutions it provides no spatial regularity, and so we also include total variation\cite{TV} regularization as a joint regularizer. Total variation has been widely studied and is well known to promote piecewise constant images with small perimeter\cite{TV,MB}. This combination, at least numerically, allows us to significantly reduce the number of projections required to obtain a good reconstruction. For practical implementation we use a smoothed version of the isotropic total variation \cite{TV}
\begin{equation}\label{isoTVMAT}
\begin{split}
\TV_{c}(a)=&\sum_{i=1}^{M^{2}-1}\sqrt{\|D_{i}a\|^{2}_{2}+c},\\				    			    
\end{split}
\end{equation}
where $c>0$ is a small smoothing constant and each $D_{i}\in\mathbb{R}^{2\times M^{2}}$ is a finite difference matrix satisfying
\begin{equation}\label{eq:Di}
D_{i}a=\begin{cases}
\begin{pmatrix}a(i)-a(i+1)\\a(i)-a(i+M)\end{pmatrix}& \mathrm{if}~1\leq i\leq M^2-M~\&~\mod(i,M)\neq 0\\				    			       
\begin{pmatrix}0\\a(i)-a(i+M)\end{pmatrix}& \mathrm{if}~1\leq i\leq M^2-M~\&~\mod(i,M)= 0\\				    			       
\begin{pmatrix}a(i)-a(i+1)\\0\end{pmatrix}& \mathrm{if}~M^2-M+1\leq i\leq M^2-1.\\				    			       
\end{cases}
\end{equation}
Note that the smoothness of the total variation is required in order to guarantee global Lipschitz continuity of its gradient.
				    			    
At this point we would like to mention that although \eqref{isoTVMAT} is convex and the non-convex multi-bang regularizer is weakly convex, we still have to be careful with the data fidelity term. It can be shown that for sufficiently large $a$, $\|R[a]f-d\|^2$ may be non-convex. Because of this we use the following alternating minimization scheme \cite{alternating} designed for non-convex objective functions
				    			     
\begin{equation}\label{Alternating1}
\begin{split}
a^{k+1}&\in\argmin_{a}\mathcal{R}(a,f^{k})+\frac{1}{2\xi^{k}}\|a-a^{k}\|^{2},\\				    			     
f^{k+1}&\in\argmin_{f}\mathcal{R}(a^{k+1},f)+\frac{1}{2\xi^{k}}\|f-f^{k}\|^{2},				    			     
\end{split}
\end{equation}
for sufficiently small $\{\xi_{k}\}_{k=1}^\infty$. We first turn our attention to the $a$ update.
				    			    
\subsection{Updating attenuation $a$}
				    			 
Since we only concern ourselves with parts of the objective function $\mathcal{R}(a,f)$ involving $a$, the $a$ update in \eqref{Alternating1} is equivalent to
\begin{equation}
\label{aupdate}
a^{k+1}\in\argmin_{a}\|R[a]f^{k}-d\|^{2}+\alpha\mathcal{M}(a)+\lambda{\TV}_{c}(a)+\frac{1}{2\xi^{k}}\|a-a^{k}\|^{2}.
\end{equation}
For the purpose of solving this optimisation problem we introduce two auxiliary variables which are $x$ corresponding to $a$ itself, and $y$ corresponding to the discrete derivative of $a$. These two parts are linked by the matrix equation 				    			    
$$Dx=y$$ 
where $D$ is the finite difference matrix obtain by stacking all the $D_{i}$ defined by \eqref{eq:Di} on top of each other. We further split $y$ into a series of 2 by 1 column vectors which are linked to $x$ by the matrix equations 				    			    
$$D_{i}x=y_{i}.$$
Therefore, we can rewrite \eqref{aupdate} as 				    			    
\begin{equation}\label{admmaup}
\begin{split}
a^{k+1}\in&\argmin_{x}\|R[x]f^{k}-d\|^{2}+\alpha\mathcal{M}(x)+\lambda\sum_{i=1}^{M^{2}-1}\sqrt{\|y_{i}\|^{2}_{2}+c}+\frac{1}{2\xi^{k}}\|x-a^{k}\|^{2},\\
&\mathrm{subject~to}~Dx=y.
\end{split}
\end{equation}
A standard algorithm for solving a optimization problem in the form of \eqref{admmaup} is the Alternating Direction Method of Multipliers (ADMM)\cite{addm}. In this case the augmented Lagrangian\cite{addm} is given by 				    			
\begin{equation}\label{augmentedLagrangian1}
\begin{split}			    			    
\fL(x,y,\mu)& =\sum_{i=1}^{M^{2}-1}\left(\lambda \sqrt{\|y_{i}\|^{2}_{2}+c} -{\mu}_{i}^{T}({y}_{i}-D_ix)+\frac{\beta}{2}\|{y}_{i}-D_{i}x\|^{2}\right)\\
& \hskip2cm+\|R[x]f^{k}-d\|_2^{2}+\alpha\mathcal{M}(x)+\frac{1}{2\xi^{k}}\|x-a^{k}\|^{2}.
\end{split}		    			    
\end{equation}
where $\beta>0$ and $\mu_{i}$ are Lagrange multipliers related to $y_i$. We also define a vector $\mu$ given by placing the $\mu_i$ related to $y_i$ in the same positions in $\mu$ as the corresponding $y_i$ are in $y$. The ADMM algorithm for solving \eqref{aupdate} then proceeds as follows
\begin{equation*}
\begin{split}
x^{l+1}&=\argmin_{x}\fL(x,y^{l},\mu^{l}),\\				    			    
y^{l+1}&=\argmin_{y}\fL(x^{l+1},y,\mu^{l}),\\				    			    
\mu^{l+1}&=\mu^{l}+\beta(y^{l+1}-Dx^{l+1}).				    			    
\end{split}
\end{equation*}
Removing terms not involving $x$, we see that the $x$ update for $x^{l+1}$ can be calculated using the first order optimality condition
\begin{equation*}
\begin{split}
0&\in\partial_{x}\left\{ \|R[x^{l+1}]f-d\|^{2}+\frac{\beta}{2}\|y^{l}-Dx^{l+1}\|^{2}+\frac{1}{2\xi^{l}}\|x^{l+1}-a^{k}\|^{2}-\mu^{T}(y^l-Dx^{l+1})+\alpha \mathcal{M}(x^{l+1})\right\}\\
0&\in \nabla_{x}(\|R[x^{l+1}]f-d\|^{2})+\beta D^{T}(Dx^{l+1}-y^l)+D^{T}\mu+\frac{1}{\xi^{l}}(x^{l+1}-a^{k})+\partial_{x}\alpha\mathcal{M}(x^{l+1}).
\end{split}
\end{equation*}
where $\nabla_{x}( \|R[x^{l+1}]f-d\|^{2})$ is determined from \eqref{DAtRT} as in \cite{Phil}. Note that since the non-convex multi-bang regularizer is separable, we have the elementwise optimality condition
\begin{equation} \label{eq:elwise}
\begin{split}
0& \in \nabla_{x}(\|R[x^{l+1}]f-d\|^{2})(i)+\beta D^{T}( Dx^{l+1}-y^l)(i)+D^{T}\mu(i)\\
&\hskip4cm +\frac{1}{\xi^{l}}(x^{l+1}(i)-a^{k}(i))+\partial_{x}\alpha{m}(x^{l+1}(i)).
\end{split}
\end{equation}
Now as $\nabla_{x}(\|R[x^{l+1}]f-d\|^{2})+\beta D^{T} (y^l -Dx^{l+1})+D^{T}\mu+\frac{1}{\xi^{l}}(x^{l+1}-a^{k})$ is differentiable it is Lipschitz continuous. Furthermore, as the pointwise multi-bang regularizer is weakly convex by \cite{weakconvex} the pointwise multi-bang regularizer admits a well-defined proximal map
\begin{equation*}
\mathrm{prox}_{\frac{1}{t}} \alpha m(x)=\begin{cases} a_{0}&\mathrm{if}~x\leq x_{0,+}\\				    			    
a_{i}&\mathrm{if}~x_{i,-}\leq x \leq x_{i,+}~\mathrm{for~}i\in\{1, 2,\ ...\ , n-1\}\\				    			    
a_{n}&\mathrm{if}~x_{n,-}\leq x\\
\frac{1}{1-2\alpha t}\bigg(x-\alpha t(a_{i+1}+a_i)\bigg)&\mathrm{if}~x_{i,+}<x<x_{i+1,-}~\mathrm{for}~i\in\{0,1,\ ...\ , n-1\}
\end{cases}
\end{equation*}
where 
\begin{equation*}
\begin{split}		 	
x_{i,-}&=a_i-\alpha t(a_{i}-a_{i-1})~\mathrm{for}~i=1,\ ...\ , n,\\				    			    
x_{i,+}&=a_i+\alpha t(a_{i+1}-a_{i})~\mathrm{for}~i=0,\ ...\ , n-1.			    			    
\end{split}
\end{equation*}
$\mathcal{A}=\{a_0,a_1,\ ...\ , a_n\}$ is the admissable set and $\frac{1}{2}>\alpha t>0$.				    			   
\begin{figure}
\centering
\includegraphics[scale=0.4]{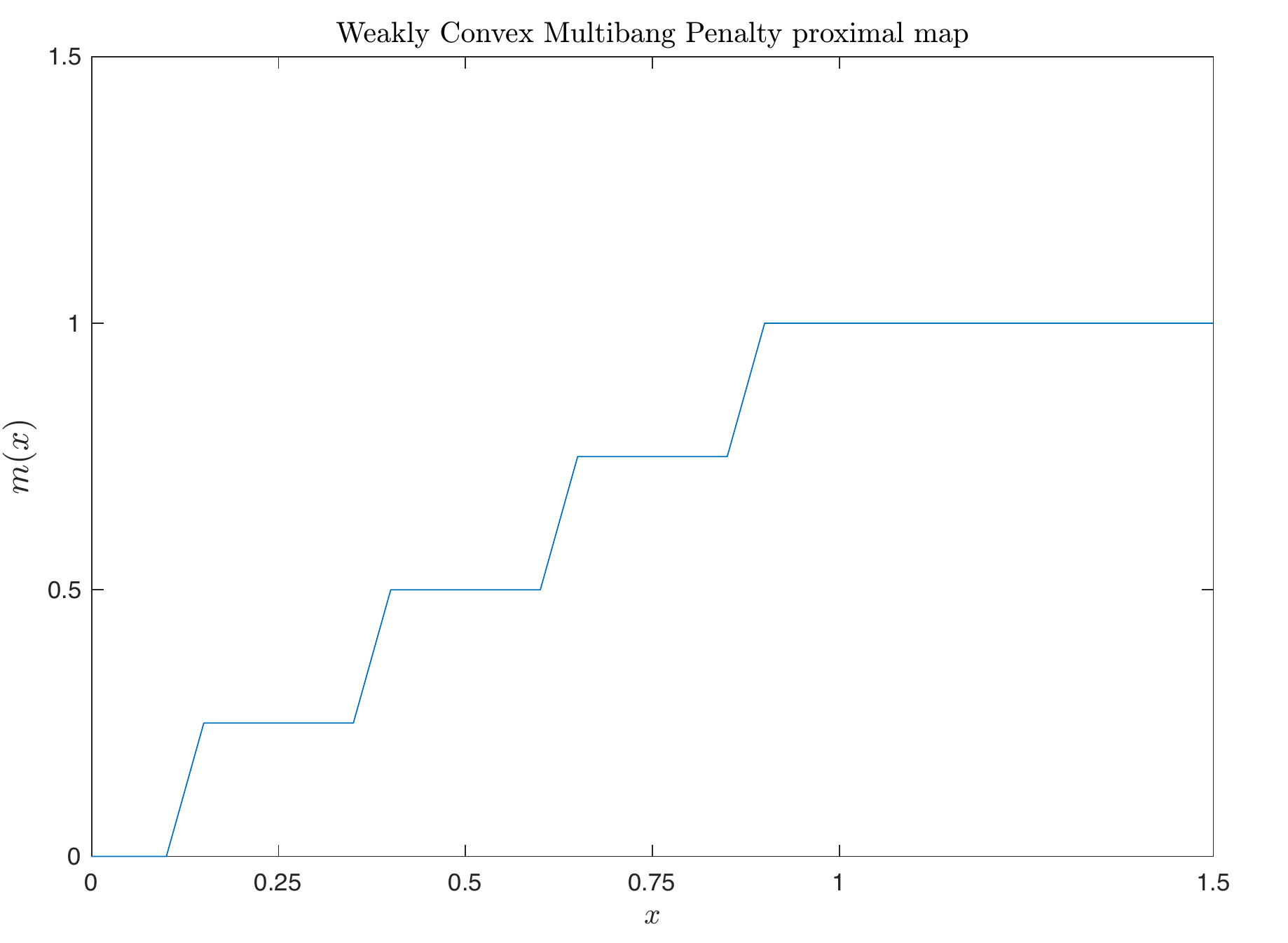}				    			    		
\caption{Weakly convex proximal map}\label{weakconvex}				    			    
\end{figure}
Figure \ref{weakconvex} gives an example of the proximal map for the weakly convex multi-bang regularizer when the admissible set is $\mathcal{A} = \{0,\ 0.25,\ 0.5,\ 0.75,\ 1\}$. Note in particular that $\mathrm{prox}_{\frac{1}{t}}\alpha {m}(a_{i})=a_{i}$ which is in contrast to the convex case \cite{MB}. Using this we can we find $x^{l+1}$ satisfying the optimality condition \eqref{eq:elwise} via a fixed point iteration such as ISTA or FISTA \cite{FISTA}. Indeed, provided $0<\alpha t<\frac{1}{2}$ by \cite{FISTA,weakconvex} both ISTA and FISTA produce iterates which converge to a solution $x^{l+1}$ of \eqref{eq:elwise} to within any prescribed tolerance.
				    			     
We now turn our attention to the $y$ update. The first order optimality condition for the $y$ update gives			
\begin{equation}\label{yupdate}			 
0=  \lambda\frac{y_{i}}{\sqrt{\|y_{i}\|^{2}_{2}+c}}-\mu_{i}+\beta(y_{i}-D_{i}x^{l+1})			 
\end{equation}
for all $i$. Whilst this cannot be explicitly solved for $y_i$ easily, we can make use of the gradient on the right hand side of \eqref{yupdate} to solve the $y$ update via gradient descent. Once we have updated all of the $y_i$ in this way we can combine them to update $y$. Finally since $\|R[a]f^{k}-d\|^{2}+\alpha\mathcal{M}(a)+\lambda{\TV}(a)$ is lower semi-continuous and $\sum_{i=1}^{M^{2}-1}\lambda{\sqrt{\|y_{i}\|^{2}_{2}+c}}$ has Lipschitz continuous gradient, \cite{Guo} gives convergence of ADMM to a critical point; that is, both the primal residual $r^{l}:=y^{l+1}-Dx^{l+1}$ and dual resdual $s_l:=\beta D^{T}(y^{l+1}-y^{l})$ converge. Numerically we can speed up the rate of convergence of the ADMM algorithm by having an adaptive $\beta$. We use the following scheme from \cite{addm}: pick $\beta^{0}>0$ then for $l\geq 0$ define
\begin{equation}\label{betaupdate}
\beta^{l+1}:=\begin{cases}
\hfill\tau^+\beta^{l}\hfill&\hfill\mathrm{if} \|r_{l}\|_2>\nu \|s_{l}\|_2\hfill\\				    			     	
\hfill\frac{\beta^{l}}{\tau^-}\hfill&\mathrm{if} \|s_{l}\|_2>\nu \|r_{l}\|_2\\				    			     	
\hfill\beta^{l}\hfill&\mathrm{otherwise}
\end{cases}
\end{equation}
for some chosen positive $\tau^{\pm}$ and $\nu$. This completes the $a$ update section of the numerical method. We now turn our attention to updating $f$.
				    			     
\subsection{Updating source radiation $f$}
				    			     
Removing terms not involving $f$, the $f$ update satisfies				    			     
\begin{equation}\label{fupdate}				    			     
f^{k+1}\in\argmin_{f}\|R[a^{k+1}]f-d\|^{2}+\eta \sum_{i=1}^{M^{2}-1}\sqrt{\|D_{i}f\|^{2}_{2}+c}+\frac{1}{2\xi^{k}}\|f-f^{k}\|^{2}
\end{equation}
again where $c>0$ is some small smoothing constant. 
We find $f^{k+1}$ solving this equation via ADMM \cite{addm,Guo} in a similar way to the $a$ update. However it is simpler here because we do not have the multi-bang regularization term. The method is again proven to converge to a critical point.
				    			    
With both the $a$ and $f$ update dealt with we are ready to outline the joint reconstruction algorithm.
\begin{algorithm}[H]
\caption{Joint reconstruction algorithm}\label{jointrecon}
\begin{algorithmic}[1]
\State Input $a^0$ as initial guess, step sizes $t,\beta^{0}$, tolerances $\delta_{1},\delta_{2},\delta_{3},\delta_4,\delta_5$ and regularization parameters $\alpha,\lambda$ and $\mu $.
\State Set $f^{0}$ to be the least squares solution of $\|R[a^{0}]f-d\|^{2}$.
\For{$k\geq 0$}
\State Set $x^0 = a^k$ and $y^0 = D x^0$.
\For{$l\geq 0$}
\State Update $x^{l+1}$ via ISTA or FISTA with $\delta_1$ as a tolerance on $\|x^{l+1}-x^{l}\|$.
\State Update $y^{l+1}$ via gradient descent on \eqref{yupdate}.
\State Set $\mu^{l+1}=\mu^{l}+\beta^{l}(y^{l+1}-Dx^{l+1})$.
\State Update $\beta^{l+1}$ via \eqref{betaupdate}
\State Terminate when $r^{l}<\delta_2$ and $s^{l}<\delta_{3}$ and output $a^{k+1} = x^{l+1}$.
\EndFor
\State Update $f^{k+1}$ via \eqref{fupdate} using ADMM with tolerance $\delta_4$.
\State Terminate when $\|a^{k+1}-a^{k}\|_{2}<\delta_5$ and $\|f^{k+1}-f^{k}\|_{2}<\delta_5$.
\EndFor
\end{algorithmic}
\end{algorithm}		
\noindent We point out that in this algorithm $\beta^0$ is reset to the same initialised value whenever the inner iterations aimed at the $a$ update in \eqref{Alternating1} (those indexed by $l$) restart. With the numerical method outlined we now present some numerical results.
    			    
\section{Numerical Reconstructions} \label{sec:numex}

Throughout this section we produce data on a 340 by 340 pixel grid and reconstruct on a 200 by 200 grid to avoid inverse crime. All of the following examples have $5\%$ added Gaussian white noise and were performed on a standard 4 core laptop using MATLAB. Note that much of the computational time is spent computing and recomputing the matrix representation of $R[a]$ when $a$ is updated, and many of the steps in this reconstruction can be done using parallel computing toolboxes. Unless otherwise stated the following reconstructions use 12 parallel ray projections which are equally spaced with some small perturbation to make the angles irrationally related (i.e. unless otherwise stated we only use data with 12 different values of $\theta$). Irrationally related angles have been shown to reduce the number of projections required to obtain good reconstructions\cite{chenrays,limrays}. For 12 projections the simultaneous reconstruction algorithm takes approximately 8 minutes.
				    		 
There are a large number of parameters to control which gives good flexibility but does require extensive parameter tuning in order to obtain optimal results. In the following examples we use  initial guesses where $a$ is constant. In practice convergence is obtained for all tested phantoms for any constant initial guess of $a$, provided the constant value lies between $a_0$ and $a_n$ in the admissable set. We therefore use initial guess $a^{0}=0$ for all numerical results presented here. The last general comment we make is that if we set $\xi=\infty$, effectively removing the added terms $\|a-a^{k}\|^{2}$ and $\|f-f^{k}\|^{2}$ from \eqref{Alternating1} we still obtain convergence. In many cases removing this part improves the speed of convergence, although the theoretical proof of convergence does not hold in this case. Throughout the section we fix $\xi=50$ and set all tolerances $\delta_i$ to $1\times 10^{-3}$.

\begin{figure}
\caption{Numerical reconstruction with joint regularizer versus total variation alone}\label{TVfig}				    		   	
\includegraphics[width=\textwidth]{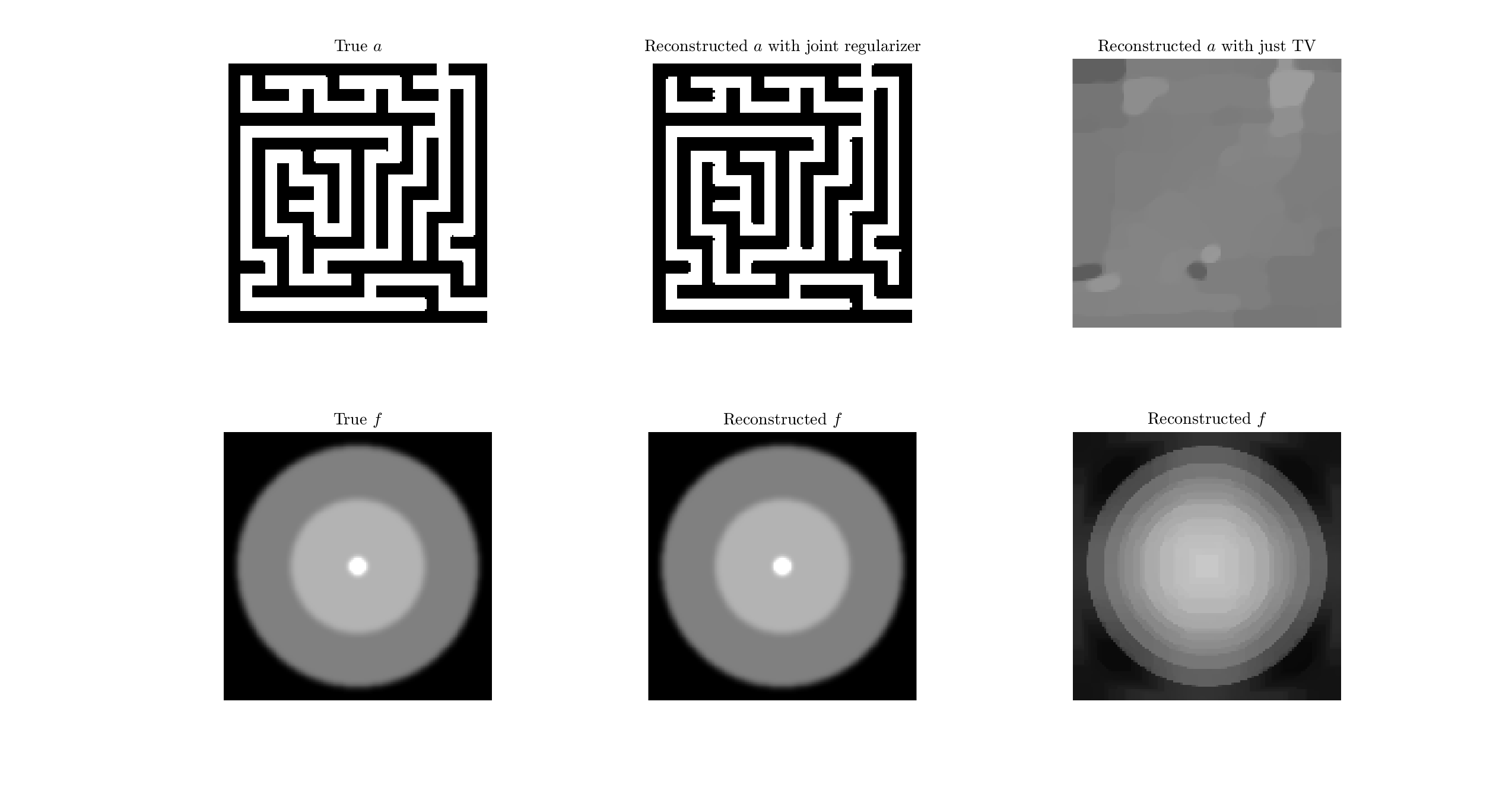}				    		   
\end{figure}				    		   
Figure \ref{TVfig} shows reconstructions obtained via an optimized parameter joint TV and multi-bang regularizer algorithm against those obtained purely by a TV approach. The left hand column gives the true phantoms for $a$ and $f$; $a$ is binary so here $\mathcal{A}=\{0,1\}$. The middle column shows the joint reconstruction for $a$ and $f$ with both TV and multi-bang regularization. Here the step sizes are $t=10$ and $\beta^{0}=0.1$. The regularization parameters are $\alpha=0.2$ and $\lambda=\eta=0.1$. The reconstruction for $a$ is multi-bang with the overall structure very well recovered. This is also seen in the reconstruction for $f$. We note that an L1 regularizer could have been used in the place of TV on $f$. In practice however TV performs much more favourably in removing cross talk-artefacts which are common in these types of joint reconstructions\cite{Bukgeim,Quinto,GourionNoll}. The right hand column is an optimized parameter reconstruction obtained using just TV with 80 projections. In this case both the binary nature and the structure of $a$ are lost, even with the extra data. This can also be seen in the recovery of $f$; the structure is roughly recovered but there are several ring artefacts appearing in the reconstruction.

\begin{figure}
\caption{Numerical reconstructions showing effect of the number of projections used  }\label{projection}				    		   	
\includegraphics[width=\textwidth]{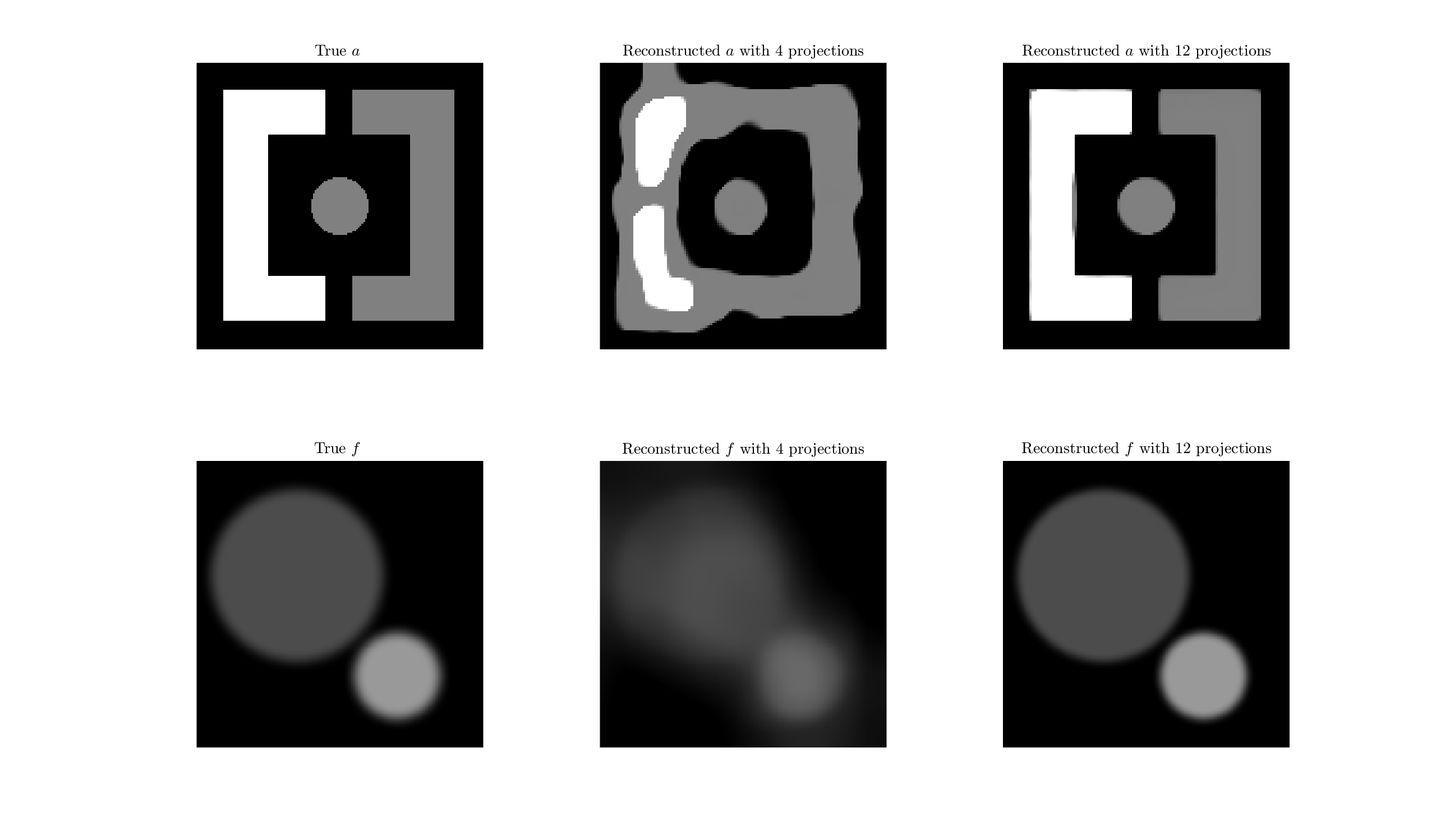}				    		   
\end{figure}				    		   
Figure \ref{projection} shows the effect of the number of projections on reconstruction quality. Here the phantom is made up of 3 regions with $\mathcal{A}=\{0,0.5,1\}$. The left column shows the true phantoms for $a$ and $f$. The middle column is an optimized reconstruction using 6 projections with $\alpha=0.1$ and $\lambda=\eta=0.05$. The right column shows an optimized reconstruction using 12 projections with $\alpha=0.1,\lambda=0.05$ and $\eta=0.15$. In both reconstructions for $a$ we obtain multi-bang solutions. The middle column shows a poor recovery of the structure of $a$ and $f$. The recovered $a$ has a lot of misclassification and has been unable to separate the regions. The inaccuracies in $a$ have an impact on the recovery of $f$, with the outer most regions of $f$ being poorly recovered. The rightmost column is a very good recovery of both $a$ and $f$, with just a small section on the left bracket being misclassified. The matching $f$ is also very well recovered. Although the smoothed TV regularization removes cross talk artefacts it is important not to use too large $\eta$ as this removes the continuous nature of $f$ at the edges. We also remark that there is no significant improvement in the quality of the reconstruction if we increase the number of projections further.

\begin{figure}
\caption{Graph of percentage multi-bang pixels at each iteration}\label{mbvaluegraph}				    		   	
\begin{center}
\includegraphics[scale=0.2]{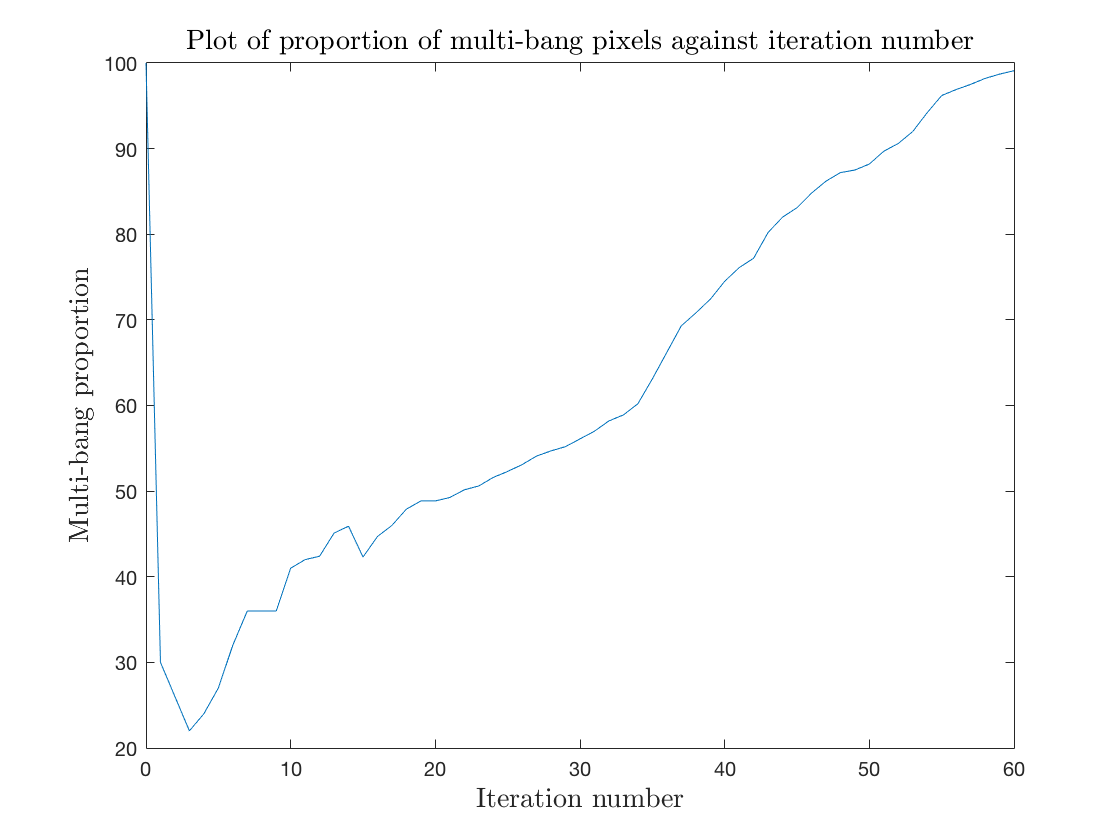}				    		   	
\end{center}
\end{figure}
Figure \ref{mbvaluegraph} shows a plot of the proportion of pixels taking admissible values against outer iteration number (that is $k$ in Algorithm \ref{jointrecon}) for the rightmost reconstruction of $a$ in Figure \ref{projection}. The initial guess is $a^0$ constant at 0, which is why the initial proportion is 1. The proportion increases monotonically after about 20 iterations with some large jumps before this point. These typically line up with $\beta^{l+1}$ being either increased or decreased in the inner iterations. This graph is typical for  reconstructions presented here and suggests that another suitable stopping criteria would be to terminate after a certain proportion of admissible values is reached. Typically a convergent reconstruction has a multi-bang proportion of over 0.95 by the time the algorithm is terminated by the step size tolerances.

\begin{figure}
\centering
\caption{Numerical reconstructions showing the effect of shrinking the size of $\mathrm{supp}(f)$.}\label{fshrink}				    		   
\includegraphics[width=.8\textwidth]{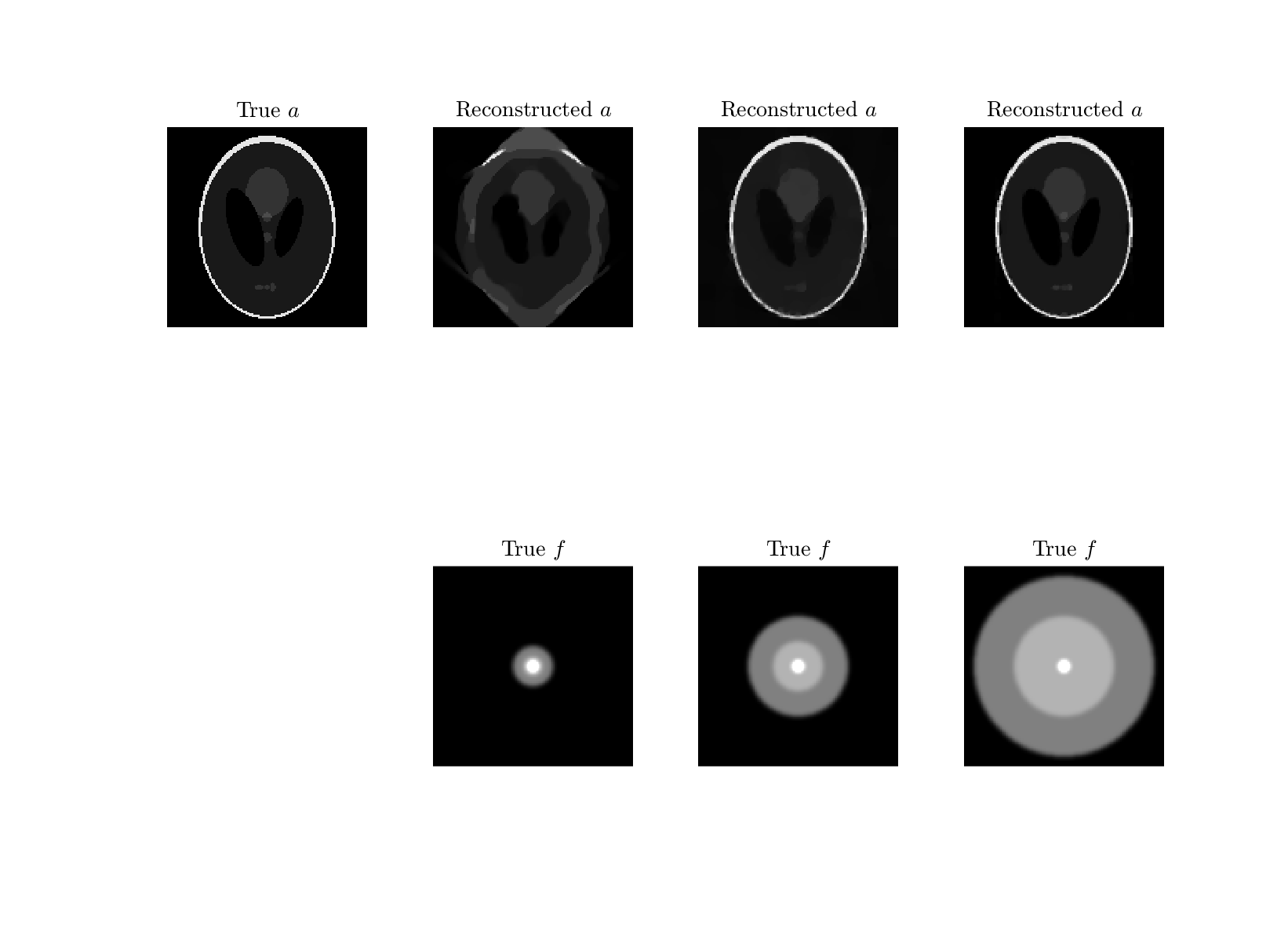}				    		   
\end{figure}
Figure \ref{fshrink} shows the effect of shrinking the size of the support of $f$ on reconstruction. This is linked to the proof of uniqueness from Theorems 1 and 2, where we require $f$ to be non-zero on a sufficient number of rays tangent to the jumps in $a$. Here the true $a$ is a multi-bang version of the Shepp-Logan phantom, with admissable set $\mathcal{A}=\{0,0.2,0.3,0.4,1 \}$. The step sizes are $t=0.075$ and $\beta^{0}=0.1$ and regularization parameters are $\alpha=0.1,\lambda=0.05$ and $\eta=0.15$ for each reconstruction. The top row shows reconstructions of $a$ with fixed and known true $f$ in the bottom row. The reconstruction algorithm here is then performed by simply performing one update for $a$. The rightmost reconstruction captures the shape and classifies almost all pixels correctly; most importantly it captures the smallest regions well. Note here that as in Theorem 1 and 2 $f$ has a larger support than $a$. It is possible to obtain reconstructions similar to that of the rightmost recovery for $f$ which have a slightly smaller support than $a$. The 2nd column from the right also has a decent shape recovery but has lost some of the finer features such as the ellipses at the bottom. The 2nd column from the left does a poor job of the recovery of $a$ with very few pixels correctly classified as $1$. The high valued outer ellipse is lost and all the finer details are missing. This is a similar effect to reducing the number of projections, which could be expected as reducing $f$ leads to fewer rays contributing data per projection. We can however still see some detail outside of the support of $f$, this is due to TV being able to fill in the gaps and extend our visibility. The angled straight edges in the reconstruction are related to the angles of the projections in the data set.

\begin{figure}
\caption{Numerical reconstruction of walnut phantom for $a$ with 30 projections}\label{walnut}				    		   	
\includegraphics[width=\textwidth]{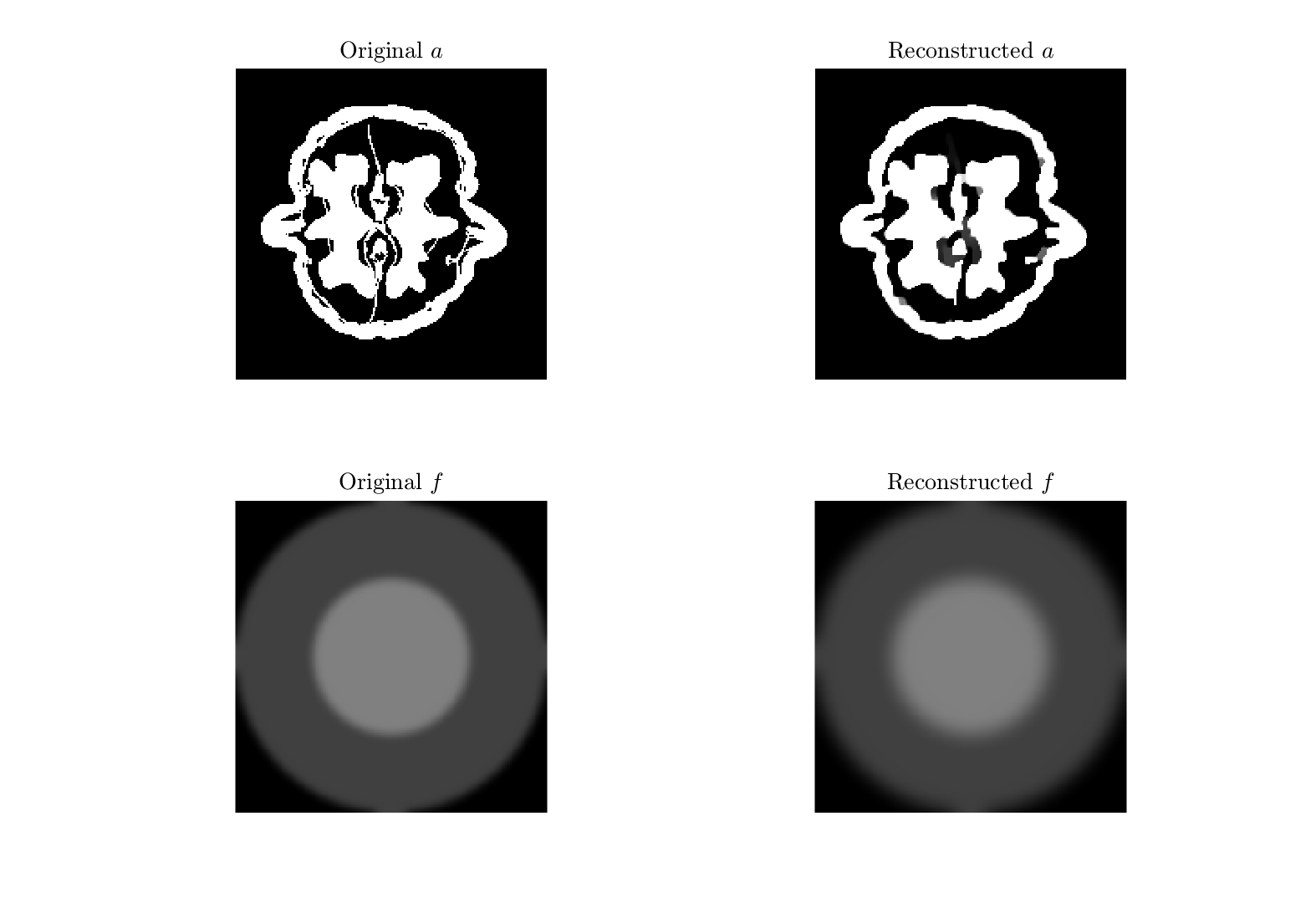}				    		   
\end{figure}
The final numerical result we present in Figure \ref{walnut} is obtained using an image of a walnut reconstructed from CT data by the Finnish Inverse Problems Society \cite{finnish} as the attenuation map. The walnut phantom for $a$ is binary and so $\mathcal{A}=\{0,1\}$. Here the step sizes are $t=0.05$ and $\beta^{0}=0.2$ and the regularization parameters are $\alpha=0.2,\lambda=0.1$ and $\eta=0.15$. The left hand column is the true $a$ and $f$ and the right hand column the reconstruction. In general the larger structures of $a$ are recovered but the finer details are lost. This is still true even if the number of projections is upped significantly. This is in part due to the detail being compressed going from 340 by 340 to 200 by 200 pixels and the TV regularizer eliminating the smallest non-zero regions. The larger sections are well-recovered and the outermost boundary is very well classified. The reconstruction for $f$ is again good with the areas towards the boundaries being the areas most affected by the errors in $a$. There are no cross-talk artefacts present, even with the more complicated $a$. 

\section{Conclusion}

In this paper we have presented and proved two theorems on the identification problem for SPECT involving multi-bang attenuation. In particular, for nicely multi-bang $a$ and $f \in C_c^1(\mathbb{R}^2)$ non-negative with sufficiently large support we have shown uniqueness of joint recovery for the AtRT. The method of proof for these theorems gives possible methods to produce similar results with further relaxed conditions on $a$ and $f$, and we intend to investigate this in future work.
				    		  
On the numerical side, we have formulated a variational problem including a weakly convex version of the convex multi-bang regularizer\cite{MB,MBorig} and a smoothed total variation, and presented an algorithm for simultaneous recovery of $a$ and $f$ from the resulting variational problem \eqref{variationalproblem}. Using an alternating direction approach for non-convex objective functions \cite{alternating} coupled with ADMM\cite{addm} we are able to successfully solve these variational problems. The addition of a joint multi-bang and total variation regularizer has produced good results for joint recovery with projection numbers similar to those used in the X-Ray recovery case when the image is known to have only a finite number of values \cite{chenrays,limrays}. The apparent convergence of the algorithm even in the case $\xi_k = \infty$ in all cases we have investigated, and also independent of the smoothing parameter $c$ for the total variation presents theoretical questions for future work. Also, even though the variational problems on which our method is based are non-convex, our algorithm consistently converges to a reasonable approximation for the correct solution. This suggests we may actually be finding the global minimiser, or be getting close to the global minimiser, and there is potential for future research investigating whether this is indeed correct.

Finally, the numerical examples shown in Figure \ref{fshrink} indicate that we are able to obtain good recovery for $a$ even when the support of $f$ is smaller than required in our theoretical results. We suspect that the total variation may be playing a role in filling in the boundaries of regions of constant $a$, and would like to investigate this further as well.

\bibliography{Multimatbib}{}
\bibliographystyle{plain}    			    

\end{document}